\documentclass[a4, 12pt]{amsart}
\usepackage{amssymb}
\usepackage{amstext}
\usepackage{amsmath}
\usepackage{amscd}
\usepackage{latexsym}
\usepackage{amsfonts}
\usepackage[all]{xy}	

\usepackage{graphicx}

\newtheorem{Theorem}{Theorem}[section]

\newtheorem{theorem}[Theorem]{Theorem}
\newtheorem{prop}[Theorem]{Proposition}
\newtheorem{pro}[Theorem]{Proposition}

\newtheorem{proposition}[Theorem]{Proposition}
\newtheorem{lem}[Theorem]{Lemma}

\newtheorem{cor}[Theorem]{Corollary}

\newtheorem{fact}[Theorem]{Fact}

\newtheorem{dfn}[Theorem]{Definition}
\newtheorem{setting}[Theorem]{Settings}

\newtheorem{question}[Theorem]{Problem}

\newtheorem{claim}{Claim}

\def\x{\underline x}

\def\Hom{\operatorname{Hom}}

\def\Ext{\operatorname{Ext}}

\newcommand{\rmg}{\mathrm{g}}

\newcommand{\calD}{\mathcal{D}}

\newcommand{\calF}{\mathcal{F}}

\newcommand{\calM}{\mathcal{M}}
\newcommand{\calN}{\mathcal{N}}

\newcommand{\fka}{\mathfrak{a}}
\newcommand{\fkb}{\mathfrak{b}}
\newcommand{\fkc}{\mathfrak{c}}

\newcommand{\fkm}{\mathfrak{m}}
\newcommand{\fkn}{\mathfrak{n}}

\newcommand{\fkp}{\mathfrak{p}}
\newcommand{\fkq}{\mathfrak{q}}

\newcommand{\fkM}{\mathfrak{M}}

\newcommand{\q}{\mathfrak{q}}
\newcommand{\m}{\mathfrak{m}}

\def\H{\operatorname{H}}

\def\depth{\operatorname{depth}}
\def\Supp{\operatorname{Supp}}
\def\Ann{\operatorname{Ann}}
\def\Ass{\operatorname{Ass}}
\def\Assh{\operatorname{Assh}}
\def\Min{\operatorname{Min}}

\def\height{\operatorname{ht}}

\def\Spec{\operatorname{Spec}}

\def\gr{\operatorname{gr}}

\begin{document}

\title{ Hilbert functions of socle ideals}

\author{Hoang Le Truong}
\address{Institute of Mathematics, VAST, 18 Hoang Quoc Viet Road 10307 Hanoi Vietnam}
\email{hltruong@math.ac.vn}

\author{Hoang Ngoc Yen}
\address{The Department of Mathematics, Thai Nguyen University of education.
20 Luong Ngoc Quyen Street, Thai Nguyen City, Thai Nguyen Province, Viet Nam.}
\email{hnyen91@gmail.com}

\thanks{
\endgraf
This research is funded by Vietnam National Foundation for Science
and Technology Development (NAFOSTED) under grant number
101.04-2014.15.
}

\date{}
\date{}

\maketitle

\begin{abstract}
In this paper, we explore a relationship between Hilbert functions and the irreducible decompositions of ideals in local rings. Applications are given to characterize the regularity,  Gorensteinness, Cohen-Macaulayness and  sequentially Cohen-Macaulayness of local rings. 

\end{abstract}

{\footnotesize \tableofcontents}

\section{Introduction}

Let $I$ be an ideal of a  Noetherian local ring $(R,\m)$ such that $R/I$ is Artinian. 
The socle of $I$ is the ideal which is defined as $I:\m$.
It is the unique largest ideal $J$ of $R$ with $J\m\subseteq I$. Moreover, it is also the
unique largest submodule of module $R/I$ which has the structure of a module over the residue field $k = R/\m R$. Therefore
$\ell_R(I:\m/I)=\dim_k(I:\m/I)$ is the minimal number of socle generators of $I$, where $\ell_R(*)$ stands for the length. 
 The minimal number of socle generators of modules are as important as the minimal number of generators of the modules, to which they are (in some sense) dual, however, in general, they are much harder to find. For a deeper discussion of socle ideals we refer the reader to \cite{ABIM}, \cite{CGHPU}, \cite{GSa1}, \cite{PU}.

The minimal number of socle generators of modules is in relative to the irreducible decompositions of modules. Irreducible ideals were already presented in the famous proof of  Noether 
that ideals in Noetherian commutative rings have a primary decomposition.  She firstly observed that the Noetherian property implied every ideal $\fka$ of $R$ can be expressed as an irredundant intersection of irreducible ideals of $R$ and the number of irreducible ideals appearing in such an expression depends only on $\fka$ and not on the expression. Let us call  the number $\calN(\fka)$ of irreducible ideals of $\fka$ that appear in an irredundant irreducible decomposition of $\fka$ the index of reducibility of $\fka$. Remember that, in the case in which $\fka=I$, $\calN(I)=\ell_R([I :_R \m]/I)$  and so the index of reducibility of $I$ is also the minimal number of socle generators of $I$. 

The minimal number of socle generators is closely related to the ideas in the theory of Gorenstein rings.  Northcott and Rees \cite [Theorem 3]{NR} proved that for all parameter ideals $\q$, the minimal number of socle generators of $\q$ is $1$ then $R$ is Gorenstein. In 1957 D. G. Northcott \cite [Theorem 3]{No} proved that for parameter ideals $\q$ in a Cohen-Macaulay local ring $R$,  the minimal number of socle generators of $\q$ is constant and independent of the choice of $\q$. However, this property of constant the minimal number of socle generators for parameter ideals does not characterize Cohen-Macaulay rings.
 The example of a non-Cohen-Macaulay local ring $R$ with $\calN(\q) = 2$  for every parameter ideal $\q$ was firstly given in 1964 by S. Endo and M. Narita \cite {EN}.  In 1984 S. Goto and N. Suzuki \cite{GS1} explored, for a given finitely generated $R$-module $M$, the supremum $\sup_{\q}\calN(\q)$, where $\q$ runs through parameter ideals of $R$ and showed that the supremum is finite, when $R$ is a generalized Cohen-Macaulay module. Compared  with  the  case of  rings and modules  with  finite  local  cohomologies,  the  general case  is  much more  complicated  and  difficult  to  treat. No standard  induction techniques  work.  However, from this point of view, a natural question is how to characterize Cohen-Macaulayness of rings in term of the minimal number of socle generators. 


On the other hand, a very interesting and important numerical invariant of a graded finitely generated
$S$-module $M$ is its Hilbert function. 
Suppose that $S=\bigoplus\limits_{n\ge 0} S_n$ is positively graded and $S_0$ is an Artinian local ring. The Hilbert function of $S$ is $H_S(n) =\ell_{S_0}(S_i)$ for $i\in\Bbb N$. Hilbert’s insight was that $H_S$ is
determined by finitely many of its values. He proved that there exists a polynomial (called the Hilbert
polynomial) $h_S(t) \in \Bbb Q[t]$ such that $H_S(t) = h_S(t)$ for $t\gg 0$. 

Now let $\gr_I(R) = \bigoplus\limits_{n\ge 0}I^n/I^{n+1}$ and call it the associated graded ring of $I$. Put $S=(0):_{\gr_I(R)}\fkM$, where $\fkM=\m/I\oplus\bigoplus\limits_{n\ge 1}I^n/I^{n+1}$. Then in special cases($e_0(\m)>1$), we see that 
$\ell_R(S_n)=\ell([I^{n+1}:\m]/I^{n+1})$ for all large enough $n$. Therefore the function of  the minimal number of socle generators of $I^n$ on $n$ become polynomial when large enough $n$. In general  there exists a polynomial $p_{I}(n)$ of degree $d-1$ with rational coefficients such that  
$$\mathcal{N}(I^{n+1};R)=\ell_R([I^{n+1}  :_R \fkm]/I^{n+1} )=p_{I}(n)$$  for all large enough $n$. Then, there are integers $f_i(I)$ such that
$$p_{I}(n)=\sum\limits_{i=0}^{d-1}(-1)^if_i(I)\binom{n+d-1-i}{d-1-i}.$$
These integers  $f_i(I)$   are called the Noetherian coefficients of $I$. 
In particular, the leading coefficient  $f_0(I)$ is  called the irreducible multiplicity of $I$, by first author. It was shown that the index of reducibility of parameter ideals can be used to characterize the Cohen-Macaulayness of local rings. 
From this point of view, we
explore this notions in this paper, where we apply it to characterize the regularity,  Gorensteinness, Cohen-Macaulayness and  sequentially Cohen-Macaulayness of local rings.

Now let us recall the definition of the Hilbert-Samuel polynomial of $I$. It is well known that the Hilbert-Samuel function $\ell_R(R/I^{n+1})$ become the polynomial which is called Hilbert-Samuel polynomial
$$\ell_R(R/I^{n+1})=\sum\limits_{i=0}^{d}(-1)^ie_i(I)\binom{n+d-i}{d-i}.$$
for all large enough $n$. These integers  $e_i(I)$   are called the Hilbert coefficients of $I$. In the particular case, the leading coefficient  $e_0(I)$ is  said to be the multiplicity of $I$.
In \cite{Na}, Nagata gave a characterization of the regularity of local rings in term of  the multiplicity of the maximal ideal. The multiplicity of the maximal ideal is $1$ if and only if  $R$ is regular, provided that $R$ is unmixed, that is $\dim \hat R/\fkp=d$ for all $\fkp\in \Ass(\hat R)$. Therefore it is a natural question to ask whether one may establish a similar correspondence between the regularity of local rings and the irreducible multiplicity of the maximal ideal.
The following result is given an answer of this question.

\begin{theorem}\label{6.100}
 Assume that  $R$ is unmixed.  Then  $R$ is regular if and only if $f_0(\m)=1.$

\end{theorem}

Recall $e_1(I)$ is called by Vasconselos(\cite{V2}) the Chern coefficient of $I$. Then  the following result give a relationship between the Chern coefficient and the irreducible multiplicity.

\begin{prop}
 Assume that  $R$ is unmixed. Then for all parameter ideals $\q\subseteq \m^2$, we have
$$e_1(\q:\m)-e_1(\q)\le f_0(\q).$$
\end{prop}

In \cite{Tr2} the first author showed to characterize the Gorensteinness of rings in term of the Chern coefficient of socle parameter ideals. A local ring $R$ is Gorenstein iff $e_1(\q:\m)-e_1(\q)\le 1$ for all parameter ideals $\q\subseteq \m^2$, provided $R$ is unmixed. From this point of view, a natural question is how to characterize Gorensteinness of rings in term of  the irreducible multiplicity of parameter ideals. The following result is given an answer of this question.

\begin{theorem}\label{6.200}
 Assume that  $R$ is unmixed. Then  $R$ is Gorenstein if and only if $ f_0(\fkq)=1$,  for all parameter ideals $\frak q\subseteq \m^2$.

\end{theorem}

We denote by $r(R)=\ell_R(\Ext^d_R(R/\fkm,R))$ the Cohen-Macaulay type. The first author in \cite{Tr2} proved that if $R$ is unmixed then $R$ is Cohen-Macaulay iff $e_1(\q:\m)-e_1(\q)\le r(R)$ for all parameter ideals $\q\subseteq \m^2$. From this point of view, as in Theorem \ref{6.200}, we shall show the following result  which  is an answer of above question.

\begin{theorem}\label{6.300}
 Assume that  $R$ is unmixed. Then $R$ is Cohen-Macaulay if and only if $ f_0(\fkq)=r(R)$, for all parameter ideals $\fkq\subseteq \fkm^2$.



\end{theorem}

A natural question from Theorem \ref{6.200} and \ref{6.300}  is what happen if $R$ is not unmixed. To sate the answer of this question,  let us fix our notation and terminology.  Let $M~(\ne (0))$ be a finitely generated $R$-module 
with finite Krull dimension, say  $d =\dim_RM$. A filtration $$\calD : D_0=(0)\subsetneq D_1\subsetneq D_2\subsetneq \cdots 
\subsetneq D_{\ell}=M$$ of $R$-submodules of $M$ is  called the {\it dimension filtration} of $M$,  if  for all $1 \le i \le \ell$, $D_{i-1}$ is the largest $R$-submodule of $D_i$ with $\dim_RD_{i-1}<\dim_R D_i$, where $\dim_R(0) = - \infty$ for convention. We say that $M$ is a {\it sequentially Cohen-Macaulay} $R$-module, 
if  $C_i=D_i/D_{i-1}$   is a Cohen-Macaulay $R$-module (necessarily with $\dim_RC_i = \dim_RD_i$) for all $1 \le i \le \ell$ (\cite{Sch, St}). Hence $M$ is a sequentially Cohen--Macaulay $R$--module with $\ell = 1$ if and only if $M$ is a Cohen-Macaulay $R$-module with $\dim R/\fkp = \dim_RM$ for every $\fkp \in \Ass_RM$. We say that $R$ is a sequentially Cohen-Macaulay ring, if $\dim R < \infty$ and $R$ is a sequentially Cohen-Macaulay module over itself. 

Let $\x=x_1,x_2, \ldots, x_d$ be a system of parameters of $M$. Then $\x$ is said to be 
{\it distinguished}, if  
$$  (x_j \mid d_{i} < j \le d) D_i=(0)$$
for all $1\le i\le \ell$, where $d_i=\dim_R D_i$. A parameter ideal $\fkq$ of $M$ is called  
{\it distinguished}, if there exists a 
distinguished system $x_1,x_2, \ldots, x_d$ of parameters of $M$ such that $\fkq=(x_1,x_2, \ldots, x_d)$. Therefore, if $M$ is a Cohen-Macaulay $R$-module, every parameter ideal of $M$ is distinguished.  Let $\Lambda (M)=\{\dim_R L \mid L \ \text{is an}\ R\text{-submodule of}\ M, L \ne (0)\}.$

With this notation the main results of this paper are summarized into the following, which gives a complete generalization of the results in the Cohen-Macaulay case to those of sequentially Cohen-Macaulay rings. 

\begin{theorem}\label{6.4}
Assume that $R$ is a homomorphic image of a Cohen-Macaulay local ring. Then the following statements are equivalent.
\begin{itemize}
\item[$(i)$] $R$ is sequentially Cohen-Macaulay.

\item[$(ii)$]  There exists an integer $n$ such that for all good parameter ideals $\fkq\subseteq \fkm^n$ and $2\le j\in\Lambda(R)$, we
    have
$$ r_{j}(R)\ge(-1)^{d-j}(e_{d-j+1}(\q:\m)-e_{d-j+1}(\fkq)).$$

\item[$(iii)$]  There exists an integer $n$ such that for all distinguished parameter ideals $\fkq\subseteq \fkm^n$ and $2\le j\in\Lambda(R)$, we
    have
$$  r_{j}(R)\geq(-1)^{d-j}f_{d-j}(\q;R).$$

\end{itemize}

\end{theorem}

Later we will give some applications of these results. First, as an immediate consequence of our main result, the assumption of Theorem \ref{6.100}, \ref{6.200} and \ref{6.300} are essential.  The necessary condition of the following result was proved by N. T. Cuong, P. H. Quy and first author (\cite[Corollary 5.3]{CQT}).

\begin{theorem}\label{6.7}
 $R$ is Gorenstein if and only if for all parameter ideals $\fkq\subseteq \fkm^2$ and $n\gg 0$, we
    have
$$ \calN(\q^{n+1};R)=\binom{n+d-1}{d-1}.$$

\end{theorem}

In \cite[Theorem 5.2]{CQT}, N. T. Cuong, P. H. Quy and first author showed that $R$ is Cohen-Macaulay if and only if for all parameter ideals $\fkq\subseteq \fkm^2$ and $n\ge 0$, we
    have
$ \calN(\q^{n+1};R)=r_d(R)\binom{n+d-1}{d-1}.$ Note that the condition of Hilbert function $\calN(\q^{n+1};R)$, holding true for all $n\ge 0$, is necessary to their proof. 
The result of  Theorem 5.2 in \cite{CQT} was actually covered in the following result,
but in view of the importance of the following result we changed the condition from Hilbert function to Hilbert polynomial.

\begin{theorem}\label{6.8}
 $R$ is Cohen-Macaulay if and only if for all parameter ideals $\fkq\subseteq \fkm^2$ and $n\gg 0$, we
    have
$$ \calN(\q^{n+1};R)=r_d(R)\binom{n+d-1}{d-1}.$$

\end{theorem}

Let us explain how this paper is organized. Section 2 is devoted to a brief survey on dimension filtrations, the notion of Goto sequences and the existence of Goto sequences.  The computation of the minimal number of socle generators of a special parameter ideal of sequentially Cohen-Macaulay ring is well understood in section 3. In section 4, our aim now is to establish  a characterization of sequentially Cohen-Macaulay rings in term of the Hilbert coefficients of the socle of distinguished parameter ideals, which is a part of  Theorem \ref{6.4}.  Continuing our discussion in section 4, the section 5 will give the complete proof of Theorem \ref{6.4}. In last section, we are going to discuss  the characterizations of the regularity,  Gorensteinness and Cohen-Macaulayness  of local rings.

\section{Goto sequences}

Let $R$ be a commutative Noetherian ring, which is not assumed to be a local ring. Let $M~(\ne (0))$ be a finitely generated $R$-module 
with finite Krull dimension, say  $d =\dim_RM$. 
We put $$\Assh_RM=\{\fkp \in \Supp_RM \mid \dim R/\fkp=d\}.$$ Then $$\Assh_RM \subseteq \Min_RM \subseteq \Ass_RM.$$
Let $\Lambda (M)=\{\dim_R L \mid L \ \text{is an}\ R\text{-submodule of}\ M, L \ne (0)\}.$ We then have $$\Lambda (M)=\{\dim R/\fkp \mid \fkp \in\Ass_RM\}.$$ 
We put  $\ell=\sharp \Lambda (M)$ and number the elements  $\{d_i\}_{1\leq i\leq \ell}$ of $\Lambda (M)$  
so that  $$0 \le d_1<d_2<\cdots<d_{\ell}=d.$$
Then because the base ring $R$ is Noetherian, 
for each $1\leq i\leq \ell$ the $R$-module $M$ contains 
the largest $R$-submodule $D_i$ with $\dim_RD_i=d_i$. 
Therefore, letting $D_0=(0)$, we have the filtration 

$$\calD : D_0=(0)\subsetneq D_1\subsetneq D_2\subsetneq \cdots 
\subsetneq D_{\ell}=M$$
of $R$-submodules of $M$, which we call the dimension filtration of $M$. 
The notion of dimension filtration was firstly given 
by P. Schenzel \cite{Sch}. 
Our notion of dimension filtration is a little different from that of \cite{CC, Sch},
but throughout  this paper let us utilize the above definition. It is standard to check that $\{D_j\}_{0\leq j\leq i}$ 
(resp. $\{D_j/D_i\}_{i\leq j\leq \ell}$) 
is the dimension filtration of $D_i$ (resp. $M/D_i$) 
for every $1\leq i\leq \ell$. We put   $C_i=D_i/D_{i-1}$ for $1 \le i \le \ell$.

We note two characterizations of the dimension filtration. Let $$(0)=\bigcap_{\fkp \in \Ass_RM}M(\fkp)$$ be a primary decomposition
of $(0)$ in $M$, where $M(\fkp)$ is an $R$-submodule of $M$ 
with $\Ass_RM/M(\fkp)=\{ \fkp \}$ for each $\fkp \in \Ass_RM$. 
We then have the following.

\begin{pro}[{\cite[Proposition 2.2, Corollary 2.3]{Sch}}]\label{d1}
The following assertions hold true. 
\begin{enumerate}
\item[$(1)$] $D_i=\bigcap_{\fkp \in \Ass_RM,\ \dim R/\fkp \geq d_{i+1}}M(\fkp)$
for all $0\leq i < \ell$. 
\item[$(2)$] Let $1\leq i \leq \ell$. Then 
$\Ass_RC_i=\{ \fkp \in \Ass_RM \mid \dim R/\fkp =d_i\}$ and
$\Ass_RD_i=\{ \fkp \in \Ass_RM \mid \dim R/\fkp \leq d_i\}$. 
\item[$(3)$] $\Ass_RM/D_i=\{\fkp \in \Ass_RM \mid \dim R/\fkp 
\geq d_{i+1} \}$ for all $1\leq i < \ell$. 
\end{enumerate}
\end{pro}


We now assume that $R$ is a local ring with maximal ideal $\fkm$ and let $M$ be a finitely generated $R$-module with $d = \dim_RM \ge 1$ and $\calD =\{D_i\}_{0 \le i \le \ell}$ the dimension filtration. Let $\x=x_1,x_2, \ldots, x_d$ be a system of parameters of $M$. Then $\x$ is said to be 
{\it distinguished}, if  
$$  (x_j \mid d_{i} < j \le d) D_i=(0)$$
for all $1\le i\le \ell$, where $d_i=\dim_R D_i$. A parameter ideal $\fkq$ of $M$ is called  
{\it distinguished}, if there exists a 
distinguished system $x_1,x_2, \ldots, x_d$ of parameters of $M$ such that $\fkq=(x_1,x_2, \ldots, x_d)$. Therefore, if $M$ is a Cohen-Macaulay $R$-module, every parameter ideal of $M$ is distinguished. 
 Distinguished system of parameters exist 
and if $x_1, x_2, \ldots, x_d$ is a distinguished system of parameters of $M$, then $x_1^{n_1},x_2^{n_2}, \ldots, x_d^{n_d}$ is also a distinguished system of parameters of $M$ for all integers $n_j \ge 1$. 


\medskip
\begin{setting}\label{2.6}{\rm
Let $\x=x_1,x_2,\ldots,x_s$ be a system  of elements of $R$ and $\q_j$ denote the ideal generated by $x_1,\ldots,x_j$ for all $j=1,\ldots,s$.   }
\end{setting}

\begin{dfn}\rm
A system $\x$ of elements of $R$ is called {\it Goto sequence} on $M$, if   for all $0\le j\le s-1$ and $0\le i\le\ell$, we have the following
\begin{enumerate}
\item[$(1)$] $\Ass(C_i/\q_j C_i)\subseteq \Assh(C_i/\q_j C_i)\cup \{\fkm\}$,
\item[$(2)$] $x_jD_i=0$ if $d_i<j\le d_{i+1}$,
\item[$(3)$] $\q_{j-1}:x_j=\H^0_\m(M/\q_{j-1}M) \text{ and } x_j\not\in\fkp \text{ for all } \fkp\in\Ass(M/\q_{j-1}M)-\{\fkm\}$.
\end{enumerate}

\end{dfn}

At first glance, the definition of normal does not seem very intuitive. Once we
enter the world of sequences, however, we will see that Goto sequence has a very
nice  interpretation and properties. We will also see that Goto sequence is useful for many inductive proofs in the next sections.
Before we can give some properties of this sequence, we first need to reformulate the notion of $d$-sequences. The sequence $x_1,x_2,\ldots,x_s$ of elements of $R$ is called a
$d$-sequence on $M$ if
$$\q_i M:x_{i+1}x_j=\q_i M:x_j$$
for all $0\le i<j\le s$. The concept of a $d$-sequence is given by Huneke \cite{Hu}  and it plays an important role in the theory of Blow up
algebra, e.g. Ress algebra. 
 In the following lemma, we will give some properties of Goto sequences that will
be used in the next sections when we study the Hilbert coefficients and Noetherian coefficients.

\medskip

\begin{lem}\label{property}
Let $\x=x_1,x_2,\ldots,x_s$ form a Goto sequence on $M$. Then we have
\begin{enumerate}
\item[$(1)$] $\x$ is part of a system of parameters of $M$. 
\item[$(2)$] $\x$ is a $d$-sequence.
\item[$(3)$] If $d=s$ then $\x$ is a distinguished system of parameters of $M$.
\item[$(4)$] $x_{i+1},\ldots,x_s$ is also a Goto sequence on $M/\q_iM$.
\end{enumerate}

\end{lem}

\begin{proof}
As an immediate consequence of the definitions we have the first assertion and the third assertion. The second assertion is followed from $(vii)$ of \cite[Theorem 1.1]{T}.

\end{proof}





\begin{lem}\label{property1}	
Let $R$ be a homomorphic image of a Cohen–Macaulay
local ring. Assume that system $\x=x_1,x_2,\ldots,x_d$ of parameters form a Goto sequence on $M$. Let $N$ denote the unmixed component of $M/\q_{d-2}M$ and $d\ge 2$. If $M/N$ is Cohen-Macaulay, so is also $M/D_{\ell-1}$.

\end{lem}
\begin{proof}

We may assume that $d \ge 3$. Because the assumption of the corollary
is inherited to the module $M/\q_iM$, it is enough to prove the following
statement.

Let $x \in R$ be a Goto sequence of length one on $M$. Let $N$ denote the unmixed component of $M/x M$ and $d \ge 3$.
If $\H^i_\m(M/N)=0$ for $i\le d-2$ then $\H^i_\m(C_{\ell})=0$ for $i\le d-1$.

For a submodule $N$ of M, we denote $\overline N=(N+xM)/xM$ the submodule of $M/xM$. Since $x$ is a Goto sequence of length one on $M$,  $\Ass(C_\ell/xC_\ell)\subseteq \Assh(C_\ell/xC_\ell)\cup\{m\}$. Therefore $N/\overline D_{\ell-1}$ has a finite length. Since $\overline M/N$ is a Cohen-Macaulay module, $H^i_\fkm(M/D_{\ell-1}+xM)=0$ for all $0<i<d-1$. Therefore, we derive from the exact sequence
$$0 \to M/D_{\ell-1}\overset{.x}\to M/D_{\ell-1} \to  M/D_{\ell-1}+xM \to 0$$
the following exact sequence:
$$0 \to \H^0_\m(M/D_{\ell-1} + xM) \to \H^1_\m(M/D_{\ell-1})
\overset{.x}\to \H^1_\m(M/D_{\ell-1}) \to 0.$$
Thus  $\H^1_\fkm(M/D_{\ell-1})=0$, and so $N/\overline D_{\ell-1}=\H^0_\fkm(M/D_{\ell-1}+xM)=0$. Hence $N=\overline D_{\ell-1}$. Moreover, since $x$ is  $C_\ell=M/D_{\ell-1}$-regular and  $C_\ell/xC_\ell\cong \overline M/\overline D_{\ell-1} = \overline M/N$ a Cohen-Macaulay module, $C_\ell$  is a Cohen-Macaulay module.
\end{proof}

We now denote $r_j(M)=\ell_R((0):_{\H^j_\fkm(M)}\fkm)$ for all $j\in \Bbb Z$.

\begin{dfn}\rm
A system $\x$ of elements of $R$ is called {\it Goto sequence of type I} on $M$, if we have   
 $$r_{d-j}(M/\q_jM)\le r_{d-j-1}(M/\q_{j+1}M),$$ for all $0\le j\le s-1$.

\end{dfn}

Now, we explore the existence of Goto sequence of type I. We have divided the proof of the existence of Goto sequence into sequence of lemmas. First, we begin with the following result of S. Goto and Y. Nakamura  \cite{GN}.
 
\begin{lem}{\cite{GN}}\label{finitely}
Let $R$ be a homomorphic image of a Cohen-Macaulay local
ring and assume that $\Ass(R)\subseteq \Assh(R)\cup \{\fkm\}$. Then
$$\calF=\{\fkp\in\Spec(R)\mid\height_R(\fkp)> 1=\depth(R_\fkp)\}$$
is a finite set.
\end{lem}

The next proposition shows the existence of a special  element which is useful for the existence of Goto sequence. 

\begin{prop}\label{exists element}
Let $R$ be a homomorphic image of a Cohen-Macaulay
local ring and $I$  an $\fkm$-primary ideal of $R$. Assume that $\calF=\{M_i\}_{i=0}^{\ell}$ is a finite filtration of submodules of $M$ such that $\Ass L_i\subseteq \Assh L_i\cup\{\fkm\}$, where $L_i=M_i/M_{i-1}$.  Then there exists an element $x\in I$ 
satisfies the following conditions
\begin{enumerate}
\item[$(1)$]   $\Ass(L_i/x^n L_i)\subseteq \Assh(L_i/x^n L_i)\cup \{\fkm\}$, For all $i=0,\ldots,\ell-1$.
\item[$(2)$] $x\not\in \fkp$,  for all $\fkp\in\Ass(M)-\{\fkm\}$.
\item[$(3)$]  $(0):_{L_i}x=\H^0_\m(L_i)$ and $(0):_Mx=\H^0_\m(M)$, for all $i=0,\ldots,\ell-1$,
\end{enumerate}

\end{prop}

\begin{proof}

Set $I_i=\Ann(L_i)$, and $R_i=R/I_i$, then  $\Ass(R_i)\subseteq \Assh(R_i)\cup\{\fkm\}$ and $\dim R/I_i> \dim R/I_{i+1}$ for all  $i=0,\ldots,s-1$. Moreover, we have 
$$\Ass(R_i)=\Ass(L_i)=\{\fkp\in\Spec(R)\mid\fkp\in\Ass(M)\text { and } \dim R/\fkp=\dim R/I_i=d_i\}\cup\{\fkm\}.$$
Set  $$\calF_i=\{\fkp\in\Spec(R)\mid I_i\subset\fkp \text{ and } \height_{R_i}(\fkp/I_i)> 1=\depth((L_i)_\fkp)\}.$$
By Lemma \ref{finitely} and the fact $\Ass(L_i)\subseteq\Assh(L_i)\cup\{\m\}$, we see that the set $$\{\fkp\in\Spec(R_i)\mid \height_{R_i}(\fkp)> 1=\depth((L_i)_\fkp)\}$$
is finite, and so that $\calF_i$ are a finite set for all $i=1,\ldots,\ell$. Put $\calF=\Ass(M)\cup\bigcup\limits_{i=1}^t\calF_i\setminus\{\fkm\}$. By the Prime Avoidance Theorem, we can choose $y\in I$ such that   $y\not\in\bigcup\limits_{\fkp\in\calF }\fkp$ and $\dim M_i/y M_i=\dim M_i-1$ for all $i=1,\ldots,\ell$.  On the other hand,  
we can choose an integer $n_0$ such that $(0):_M y^n=(0):_M y^{n_0}$ and $(0):_{L_i} y^n=(0):_{L_i} y^{n_0}$, for all $n\ge n_0$ and $i=1,\ldots,\ell$. Put $x=y^{n_0+1}$. Then we have $x\not\in\bigcup\limits_{\fkp\in\calF }\fkp$ and  $(0):_{L_i} x^2=(0):_{L_i} x$ for all $i=1,\ldots,\ell$. 
Now we show that $x$ have the conditions as required.  

First let us  prove the condition $(1)$. To this end, consider $\fkp\in\Ass(N_i/x N_i)$ with $\fkp\not=\fkm $. Then we have $\depth(L_i/xL_i)_\fkp=0$. 
 Hence $\depth(L_i)_\fkp=1$. It implies that $\height_{R_i}(\fkp)=1$, since $\fkp\not\in\calF_i$. By the assumption $R_i$ is a catenary ring, therefore 
$$\dim R/\fkp= \dim R_i-\height_{R_i}(\fkp)=\dim R_i/x R_i=\dim L_i/x L_i.$$ Hence $\fkp\in \Assh (L_i/x L_i)$.

Since the condition $(2)$ is trivial, it remains to prove the condition $(3)$. Take $\fkp\in\Ass_R(0):_{L_i}x$ with $\fkp\not=\fkm$. Hence $((0):_{L_i}x)_\fkp=(0)$ and this is a contradiction. It implies that $(0):_{N_i}x$ is finite length. Since $(0):_{L_i} x^2=(0):_{L_i} x$, we have $(0):_{L_i} x=\H^0_\fkm(L_i)$. It follows from the following exact sequence
$$0\to M_{i-1}\to M_i\to L_{i}\to 0$$
and $x\H^0_\fkm(L_i)=0$ for all $i=1,\ldots,\ell$ that the following sequence
$$0\to \H^0_\fkm(M_{i-1})\to \H^0_\fkm(M_i)\to \H^0_\fkm(L_{i})$$
$$\text{and  } 0\to (0):_{M_{i-1}}x\to (0):_{M_{i}}x\to (0):_{L_{i}}x $$
are exact. By induction and $(0):_Mx=(0):_Mx^2$, we have $(0):_Mx=\H^0_\fkm(M)$ and this completes the proof.
\end{proof}

The existence of Goto sequence is established by our next Corollary.

\begin{cor}\label{exists001}
Assume that $R$ is a homomorphic image of a Cohen-Macaulay
local ring and $I$  an $\fkm$-primary ideal of $R$.  Then there exists a system $\x=x_1,x_2,\ldots,x_s$ of elements of $I$ such that $\x$ is a Goto sequence on $M$.

\end{cor}

\begin{proof}
We prove this by induction on $s$, the case in which $s=1$ having been dealt  with in Lemma \ref{exists element}. So we suppose that $s=j\ge 2$ and that the result has been proved for smaller values of $s$. Suppose that $d_{i}<j\le d_{i+1}$ for some $i$. We see immediately from this induction hypothesis that
$$\Ass(N_i/\q_{j-1} N_i)\subseteq \Assh(N_{j-1}/\q_{j-1}N_i)\cup \{\fkm\},$$
where $\q_{j-1}=(x_1,\ldots,x_{j-1})$, for all $i=0,\ldots,\ell-1$. Moreover the sequence $x_1,x_2,\ldots,x_{d_i}$ is a system of  parameters  of $D_i$. Therefore $\Ann(D_i)+\q_{j-1}$ is $\m$-primary ideals. So that, by Lemma \ref{exists element}, there exists an element $x_j\in I\cap \Ann(D_i)$, as required. This completes the inductive step, and the proof.
\end{proof}

Let $\fkq = (x_1, x_2,\ldots,x_d)$ be a parameter ideal in $R$ and let $M$ be an $R$-module. For
each integer $n\geq 1$ we denote by $\x^n$ the sequence $x^n_1, x^n_2,\ldots,x^n_d$. Let $K^{\bullet}(x^n)$ be the
Koszul complex of $R$ generated by the sequence $\x^n$ and let
$H^{\bullet}(\x^n;M) = H^{\bullet}(\Hom_R(K^{\bullet}(\x^n),M))$
be the Koszul cohomology module of $M$. Then for every $p\in\Bbb Z$ the family $\{H^p(\x^n;M)\}_{n\ge 1}$
naturally forms an inductive system of $R$-modules, whose limit
$$H^p_\fkq=\lim\limits_{n\to\infty} H^p(\x^n;M)$$
is isomorphic to the local cohomology module
$$H^p_\fkm(M)=\lim\limits_{n\to\infty} \Ext_R^p(R/\fkm^n,M)$$
For each $n\geq 1$ and $p \in\Bbb Z$ let $\phi^{p,n}_{\x,M}:H^p(\x^n;M)\to H^p_\fkm(M)$ denote the canonical
homomorphism into the limit.


\begin{dfn}[\cite{GSa1} Lemma 3.12]\label{sur}
{\rm Let $R$ be a Noetherian local ring with the maximal ideal $\fkm$ and $
\dim R=d \ge1$. Let $M$ be a finitely generated $R$-module. Then there exists an integer $n_0$ 
such that for all systems of parameters $\x=x_1,\ldots,x_d$  for $R$ contained in $\fkm^{n_0}$ and for all $p\in \Bbb Z$, the canonical homomorphisms
$$\phi^{p,1}_{\x,M}:H^p(\x,M)\to H^p_\fkm(M)$$
into the inductive limit are surjective on the socles.
The least integer $n_0$ with this property is called a Goto number of $R$-module $M$ and denote by $\rmg(M)$.}
\end{dfn}


With this notation we have the following result.

\begin{lem}[\cite{GS1}, Lemma 1.7]\label{split}  Let  $M$  be a finitely generated $R$-module and $x$ an $M$-regular element and $\x=x_1,\ldots,x_r$ be a system of elements in $R$ with $x_1= x$. Then there exists a splitting  exact sequence for each $p \in\Bbb Z$,
$$0\to H^p(\x;M)\to H^p(\x;M/xM)\to H^{p+1}(\x;M)\to0.$$ 
\end{lem}


\begin{lem}\label{2.7}
 Let $M$ be a finitely generated $R$-module. Assume that $x$ is an $M$-regular element of $M$ such that $x\in \m^{\rmg(M)}$. Then
we have $$\rmg(M/xM)\le \rmg(M),$$
and $$r_i(M)\le r_{i-1}(M/xM)$$
for all $i\in \Bbb Z$
\end{lem}

\begin{proof}

Let $x_2,\ldots,x_d$ be a system of parameters of module $M/xM$ such that $x_i\in\m^{\rmg(M)}$. Put $\x=x_1,x_2,\ldots,x_d$ and $\q=(\x)$, where $x_1=x$. Since $x\in\m^{\rmg(M)}$, we have $\q\subseteq \m^{\rmg(M)}$.
By the definition of Goto number, we have  the canonical homomorphism
$$H^i(\x,M)\to H^i_\fkm(M)$$
into the inductive limit are surjective on the socles, for each $i\in \Bbb Z$. By the regularity of $x=x_1$ on $M$, it follows from the following sequence 
$$\xymatrix{0\ar[r]&M\ar[r]^{.x}&M\ar[r]&M/xM\ar[r]&0}$$
that there are induced the diagram
$$\xymatrix{0\ar[r]&H^i(\x;M)\ar[d]\ar[r]&H^i(\x,M/xM)\ar[r]\ar[d]&H^{i+1}(\x;M)\ar[r]\ar[d]&0\\
\ar[r]&H^i_\fkm(M)\ar[r]&H^i_\fkm(M/xM)\ar[r]&H^{i+1}_\fkm(M)\ar[r]&}$$
commutes, for all $i\in \Bbb Z$. 
 It follows from the above commutative diagrams and Lemma \ref{split} that after applying the functor $\Hom(k,*)$, we obtain the commutative diagram 
$$\xymatrix{\Hom(k,H^i(\x,M/xM))\ar[r]\ar[d]&\Hom(k,H^{i+1}(\x;M))\ar[r]\ar[d]&0\\
\Hom(k,H^i_\fkm(M/xM))\ar[r]&\Hom(k,H^{i+1}_\fkm(M))}$$
for all $i\in \Bbb Z$. Since the map $\Hom(k,H^{i+1}(\x;M))\to\Hom(k,H^{i+1}_\fkm(M))$ is surjective, so is the map $\Hom(k,H^i_\fkm(M/xM))\to\Hom(k,H^{i+1}_\fkm(M))$. Therefore the map $\Hom(k,H^{i}(\x;M/xM))\to\Hom(k,H^{i}_\fkm(M/xM))$ is surjective and $r_i(M)\le r_{i-1}(M/xM)$
for all $i\in \Bbb Z$. Thus for all systems $\x$ of parameters of module $M/xM$ such that $x_i\in\m^{\rmg(M)}$, we have
the map $\Hom(k,H^{i}(\x;M/xM))\to\Hom(k,H^{i}_\fkm(M/xM))$ is surjective for all $i\in \Bbb Z$. Hence we have
we have $$\rmg(M/xM)\le \rmg(M),$$
as required.

\end{proof}

\begin{cor}\label{leCMt}
 Let $M$ be a finitely generated $R$-module with $\dim M\ge 2$. Then there exists an integer $n$ such that for all parameter elements $x\in \m^n$, we have
$$r_d(M)\le r_{d-1}(M/xM).$$
\end{cor}
\begin{proof}
Since $\dim M\ge 2$ and $x$ is a parameter element of $M$, we have $\H^d_\m(M)= \H^d_\m(M/\H^0_\m(M))$ and $\H^{d-1}_\m(M/xM)=\H^{d-1}_\m(M/xM+\H^0_\m(M))$. Therefore we have been working under the assumption that $\depth M\ge 0$. Then by Lemma \ref{2.7}, we have $$r_d(M)\le r_{d-1}(M/xM),$$
and the proof is complete.

\end{proof}


Addition, the existence of Goto sequence of type I is established by our next Proposition.

\begin{prop}\label{exists}
Assume that $R$ is a homomorphic image of a Cohen-Macaulay
local ring and $I$  an $\fkm$-primary ideal of $R$.  Then there exists a system $\x=x_1,x_2,\ldots,x_s$ of elements of $I$ such that $\x$ is a Goto sequence of type I on $M$.

\end{prop}
\begin{proof}
We shall now show the our result by induction on $s$. In the case in which $s=1$ there is nothing to prove, because of the Lemma \ref{exists element}. So we suppose, inductively, that $s=j> 1$ and the results have both been proved for smaller values of $s$. Suppose that $d_{i}<j\le d_{i+1}$ for some $i$. By induction we have system $x_1,\ldots,x_{j-1}$ of $R$ such that satisfies the following conditions
\begin{enumerate}
\item[$(1)$] $\Ass(N_i/\q_{j-1} N_i)\subseteq \Assh(N_j/\q_{j-1}N_i)\cup \{\fkm\}$, where $\q_{j-1}=(x_1,\ldots,x_{j-1})$, for all $i=0,\ldots,\ell-1$. 
\item[$(2)$] The sequence $x_1,x_2,\ldots,x_{d_i}$ is a system of  parameters  of $D_i$.
\end{enumerate}
Let $\overline R=R/\q_{j-1}$ $\overline M=M/\q_{j-1}M$ $\fkn=\m/\fkq_{j-1}$. It follows from Corollary \ref{leCMt} that there exists an integer $n$ such that for all $x\in \m^n$ we have $r_{d-j+1}(\overline M)\le r_{d-j}(\overline M/x\overline M)$.
Put $J=(\Ann(D_i)+\q_{j-1})\cap I\cap \m^n$. Then $J\overline R$ is an $\fkn$-primary ideal of $\overline R$. 
 By Lemma \ref{exists element}, we can choose $x_{j+1}\in\Ann(D_i)\cap I\cap\m^n$, as required. With this observation, we can complete the inductive step and the proof.
\end{proof}

\section{Socle polynomial}
In this section, we introduce the notion of Noetherian coefficients and the its computation   in the sequentially Cohen-Macaulay cases.
 Recall, we say that an $R$-submodule $N$ of $M$ is irreducible if $N$ is not written as the intersection of two larger $R$-submodules of $M$. Every $R$-submodule $N$ of $M$ can be expressed as an irredundant intersection of irreducible $R$-submodules 
of $M$ and the number of irreducible $R$-submodules appearing in such an expression depends only on $N$ and not on the 
expression. Let us call, for each $\frak m$-primary ideal $I$ of $M$, the number $\mathcal N(I;M)$ of irreducible $R$-
submodules of $M$ that appearing in an irredundant irreducible decomposition of $I M$ the index of reducibility of $M$ with 
respect to $I$. Remember that $$\mathcal{N}(I;M)=\ell_R([I M :_M \fkm]/I M).$$ 
Moreover, by Proposition 2.1 \cite{CQT}, it is well known that there exists a polynomial $p_{I,M}(n)$ of degree $d-1$ with rational coefficients such that  $$\mathcal{N}(I^{n+1};M)=\ell_R([I^{n+1} M :_M \fkm]/I^{n+1} M)=p_{I,M}(n)$$  for all large enough $n$. Then, there are integers $f_i(I;M)$ such that
$$p_{I,M}(n)=\sum\limits_{i=0}^{d-1}(-1)^if_i(I;M)\binom{n+d-1-i}{d-1-i}.$$
These integers  $f_i(I;M)$   are called the Noetherian coefficients of $M$ with respect to $I$. 
In particular, the leading coefficient  $f_0(I; M)$ is  called the irreducible multiplicity of $M$ with respect to $I$. 
When $M=R$, we abbreviate $f_0(I;M)$ to $f_0(I)$. The following result will be necessary in the computation of the Noetherian coefficients of distinguished parameter ideals.

\medskip

\begin{lem}\label{2.101}
Let $N$ be a submodule of $M$ such that $M/N$ is Cohen-Macaulay and $\dim N<\dim M$. Assume that $\q$ is a parameter ideal generated by $x_1,\ldots,x_d$ such that
$$[N+\q M]:_M \fkm=N + [\q M:_M\fkm].$$
Let  $0\le s\le d$ and $\fkb=(x_1,\ldots,x_s)$. Then we have 
$$[\q^n M+N+\fkb]:_M \fkm=[\q^n M:_M\fkm]+N+\fkb,$$
for all $n\ge 0$.
\end{lem}

\begin{proof}
We put $\calM= M/\fkb M+N$ and we denote $\operatorname{gr}_{\q}(\calM)=\bigoplus\limits_{n\ge 0}\frak q^n\calM/\frak q^{n+1}\calM$. Since $\calM$
is a Cohen-Macaulay $R$-module and $\q$ is a parameter ideal of $R$-module $\calM$, sequence $x_{s+1},\ldots,x_d$ is an $\calM$-regular. Since $\calM$ is Cohen-Macaulay, we have a natural isomorphism of graded modules

$$gr_\q(\calM)=\bigoplus\limits_{n\ge 0}\frak q^n\calM/\frak q^{n+1}\calM\to \calM/\frak q \calM[T_{s+1},\ldots,T_d],$$ where $T_{s+1},\ldots,T_d$ are indeterminates.
This deduces $R$-isomomorphisms on graded parts
$$\frak q^n\calM/\frak q^{n+1}\calM\to\big( \calM/\frak q \calM[T_{s+1},\ldots,T_d]\big)_n\cong \calM/\frak q \calM^{\binom{n+d-(s+1)}{d-(s+1)}}$$ for all $n\geq 0$. On the other hand, since $\frak q$ is a parameter ideal of a Cohen-Macaulay modules $\calM$, $\frak q^{n+1}\calM:\frak m\subseteq \frak q^{n+1}\calM:\frak q=\frak q^{n}\calM$. It follows that
$\frak
q^{n+1} \calM:\frak m=\frak q^{n}\calM(\frak q \calM:_\calM\frak m)$. So we have
$$[\frak q^{n+1}M+N+\fkb M]:\frak m = \frak q^n([\frak qM+N+\fkb]:\frak m)+N+\fkb$$
because $\fkb\subseteq\fkq$.
 Since $[N+\q M]:_M \fkm=N + [\q M:_M\fkm]$, therefore we have
$$[\frak q^{n+1}M+N+\fkb M]:\frak m \subseteq \frak q^n(\frak q M:\frak
m)+N+\fkb M\subseteq \frak q^{n+1}M:\frak m+N+\fkb M.$$ Thus
$[\frak q^{n+1}M+N+\fkb M]:\frak m = \frak q^{n+1}M:\frak
m+N+\fkb M$. Hence
$$[\q^n M+N+\fkb M]:_M \fkm=[\q^n M:_M\fkm]+N+\fkb M,$$
for all $n\ge 0$.

\end{proof}

The notion of a sequentially Cohen-Macaulay module was introduced firstly by Stanley \cite{St} for the graded case and in \cite{Sch} for the local case.

\begin{dfn}[\cite{Sch, St}]\rm
 Let $\calD = \{D_i\}_{0 \le i \le \ell}$ be the dimension filtration of $M$.
We say that $M$ is a {\it sequentially Cohen-Macaulay $R$-module}, 
if $C_i$ is a Cohen-Macaulay $R$-module for all 
$1\leq i \leq \ell$, where $C_i=D_i/D_{i-1}$. We say that $R$ is a {\it sequentially Cohen-Macaulay ring}, if $\dim R < \infty$ and $R$ is a sequentially Cohen-Macaulay module over itself. 

\end{dfn}

We maintain the following settings.

\begin{setting}\label{2.6}{\rm
Let $M$ be a sequentially Cohen-Macaulay $R$-module, $d = \dim M \ge 1$, and $\calD = \{D_i\}_{0 \le i \le \ell}$ the dimension filtration. We put  $N= D_{\ell -1}$, $L= M/D_{\ell -1}$ and  choose  a distinguished system $x_1, x_2, \ldots, x_d$ of parameters of $M$ such that
$$\calN(\fkq;M)=\sum\limits_{j \in \Bbb Z}r_j(M).$$
where  $\q = (x_1, x_2, \ldots, x_d)$.}
\end{setting}

\begin{fact}\rm(See \cite
{CGT1}) \label{F2.5} 
The following assertions hold true.
\begin{enumerate} 
\item[$({\rm 1})$] Module $N$ is sequentially Cohen-Macaulay and $L$ is Cohen-Macaulay.

\item[$({\rm 2})$]  We have $[N+\q M]:_M \fkm=N + [\q M:_M\fkm]$. 
\item[$({\rm 3})$] The parameter ideal $\q$ is also a distinguished parameter ideal of $N$ such that
$$\calN(\fkq;N)=\sum\limits_{j \in \Bbb Z}r_j(N).$$

\item[$({\rm 4})$] Let $M=R$. If $e_0(\m;R)>1$ or $\q\subseteq \m^2$  then we have $I^2=\q I$, where $I=\q:\m$.
\end{enumerate}

\end{fact}

\begin{proposition}\label{P2.7}
We have
$$\mathcal{N}(\frak q^{n+1};M) = \sum\limits_{i=1}^dr_i(M)\binom{n+i-1}{i-1}+r_0(M)$$
for all $n\ge 1$.
\end{proposition}

\begin{proof}

We denote $\operatorname{gr}_{\q}(L)=\bigoplus\limits_{n\ge 0}\frak q^nL/\frak q^{n+1}L$. Since $L$
is a Cohen-Macaulay $R$-module and $\q$ is a parameter ideal of $R$-module $L$, sequence $x_1,\ldots,x_d$ is an $L$-regular. Since $L$ is Cohen-Macaulay, we have a natural isomorphism of graded modules

$$gr_\q(L)=\bigoplus\limits_{n\ge 0}\frak q^nL/\frak q^{n+1}L\to L/\frak q L[T_1,\ldots,T_d],$$ where $T_1,\ldots,T_d$ are indeterminates.
This deduces $R$-isomomorphisms on graded parts
$$\frak q^nL/\frak q^{n+1}L\to\big(L/\frak q L[T_1,\ldots,T_d]\big)_n\cong L/\frak q L^{\binom{n+d-1}{d-1}}$$ for all $n\geq 0$. On the other hand, since $\frak q$ is a parameter ideal of a Cohen-Macaulay modules $L$, $\frak q^{n+1}L:\frak m\subseteq \frak q^{n+1}L:\frak q=\frak q^{n}L$. It follows from $\ell(\q L:\m/\q L)=r_d(M)$ that
$$\ell(\frac{\q^{n+1}L:\m}{\q^{n+1}L})=\ell(\frac{\q L:\m}{\q L})\binom{n+d-1}{d-1}=r_{d}(M)\binom{n+d-1}{d-1}.$$
Since the parameter ideal $\frak q$ is good and $L$ is Cohen-Macaulay, the
following exact sequence
$$0\to N\to M\to L\to 0$$
induces the following exact sequence
$$0\to N/\frak q^{n+1}N\to M/\frak q^{n+1}M\to L/\frak q^{n+1} L\to 0.$$
It follows from  $[\frak
q^{n+1}M+N]:\frak m= \frak q^{n+1}M:\frak m+N$, by the Fact \ref{F2.5} and lemma \ref{2.101}, that by
applying $\Hom_R(k,*)$, we obtain the following exact sequence
$$0\to \Hom_R(k,N/\frak q^{n+1} N)\to \Hom_R(k,M/\frak q^{n+1}M)\to \Hom_R(k,L/\frak q^{n+1} L)\to0$$
Therefore we get that
$$\ell_R([\frak q^{n+1}M:\frak m]/\frak q^{n+1}M)=\ell_R([\frak q^{n+1} N:\frak m]/\frak q^{n+1} N)+\ell_R([\frak q^{n+1} L:\frak m]/\frak q^{n+1} L).$$
A simple inductive argument therefore shows that
$$\mathcal{N}(\frak q^{n+1};M) = \ell ([\frak q^nM : \frak m]/\frak q^nM) =
 \sum\limits_{i=1}^dr_i(M)\binom{n+i-1}{i-1}+r_0(M)$$ for all $n\ge 1$. Thus the proof
is complete

\end{proof}

We need the following result in next section.

\begin{lem}\label{3.700}
We have
$\q^{n+1}M:\m=\q^n(\q M:\m)$ for all $n\ge 0$.
\end{lem}
\begin{proof}
Since $L$ is Cohen-Macaulay and $\q$ is a parameter ideal of $L$, we have $\q^nM\cap N=\q^nN$ and
$$\q^{n+1}M:\m\subseteq [\q^{n+1}M+N]:\m=\q^n(\q M:\m)+N,$$
for all $n\ge 0$, because of the Fact \ref{F2.5} and the Lemma \ref{2.101}.
Let $a\in\q^{n+1}M:\m$ and we write $a=b+c$ for $b\in\q^n (\q M:\m)$ and $c\in N$. Then $\m c=\m (a-b)\in \q^{n+1}M\cap N=\q^{n+1}N$.  Thus $c\in \q^{n+1}N:\m$. 
Therefore $\q^{n+1}M:\m\subseteq \q^n(\q M:\m)+\q^{n+1}N:\m$. Hence
\begin{eqnarray*}
\q^{n+1}M:\m
&=&\q^n(\q M:\m)+\q^{n+1} N:\m
\end{eqnarray*}
Since $N$ is sequentially Cohen-Macaulay and $\q$ is a good parameter idea of $N$, by the induction on $\ell$, we have
$\q^{n+1} N:\m=\q^n[\q N:\m]$. Therefore we have $$\q^{n+1}M:\m=\q^n(\q M:\m),$$
as required.

\end{proof}

We close this section with the following, which  is the main result of \cite{T} of the first author.

\begin{theorem}[{\cite[Theorem 1.1]{T}}]\label{3.800}
There exists an integer $n \gg 0$ such that for every distinguished parameter ideals $\q$ of $M$ contained in $\m^{n}$, one has  the equality
$$\calN(\q; M)=\sum\limits_{j \in \Bbb Z}\ell_R((0):_{\H^j_\fkm(M)}\fkm).$$
\end{theorem}

\medskip

\section{Hilbert coefficients of socle ideals}

The purpose of this section is to give a characterization of sequentially Cohen-Macaulay rings in term of the Hilbert coefficients of the socle of distinguished parameter ideals. To discuss this, we need the concept of Hilbert coefficients.

 Let $I$ be an $\frak m$-primary ideal of a  Noetherian local ring $(R,\m)$. 
The associated graded ring $\gr_I(R) =\bigoplus_{n\geq 0}I^n/I^{n+1}$
is a standard graded ring with $[\gr_I(R)]_0 = R/I$ Artinian. Let $M$ be a finitely generated $R$-module of dimension $d$.
Therefore the associated graded module $\gr_I(M) =\bigoplus_{n\geq 0} I^nM/I^{n+1}M$ of $I$ with respect to $M$ is a finitely generated graded $\gr_I(R)$–module. The Hilbert-Samuel function of $M$ with respect to $I$ is
$$H(n)=\ell_R(M/I^{n+1}M)=\sum\limits_{i=0}^n\ell_R(I^iM/I^{i+1}M),$$
where $\ell_R(*)$ stands for the length. For sufficiently large $n$, the Hilbert-Samuel function of $M$ with respect to $I$
$H(n)$ is of polynomial type,
$$\ell_R(M/I^{n+1}M)=\sum\limits_{i=0}^d(-1)^ie_i(I,M)\binom{n+d-i}{d-i}.$$
These integers  $e_i(I,M)$   are called the Hilbert coefficients of $M$ with respect to $I$. In the particular case, the leading coefficient  $e_0(I, M)$ is  said to be the multiplicity of $M$ with respect to $I$ and $e_1(I,M)$ is called by Vasconselos(\cite{V2}) the Chern coefficient of $I$  with respect to $M$. When $M=R$, we abbreviate $e_i(I, M)$ to $e_i(I)$ for all $i=1,\ldots,s$.



\medskip


\medskip



\medskip

\noindent


\noindent

\begin{setting}\label{2.6}{\rm
Assume that $R$ is a homomorphic image of a Cohen-Macaulay local ring. Let $\calD = \{\fka_i\}_{0 \le i \le \ell}$ be the dimension filtration of $R$ with $\dim \fka_i=d_i$. We put  $S = R/\fka_{\ell -1}$ and  choose  a distinguished system $x_1, x_2, \ldots, x_d$ of parameters of $R$. Put $\q=(x_1,x_2,\ldots,x_d)$, $\fkb=(x_{d_{\ell-1}+1},\ldots,x_d)$ and $I=\q:\m$.
}
\end{setting}

In fact, the following property serves to characterize sequentially Cohen-Macaulay rings, as we will show in this section.

\medskip

\begin{prop}\label{coe}
Assume that $R$ is sequentially Cohen-Macaulay
and $$\calN(\fkq;R)=\sum\limits_{j \in \Bbb Z}r_j(R).$$ 
Then 
we have
\[e_j(I)-e_j(\q)=  f_{j-1}(\q;R) =  \left\{
\begin{array}{rl}
&(-1)^{d-1}(r_1(R)+r_0(R)), \quad \mbox{if $j=d$,} \\
& (-1)^{j-1}r_{d-j+1}(R) \hspace{1cm} \quad \mbox{otherwise,}
\end{array}
\right.\]
if $e_0(\m;R)> 1$ or $\q\subseteq \m^2$.
\end{prop}

\begin{proof}
By Fact \ref{F2.5} (4), we have $I^2=\fkq I$ and so $I^{n+1}=\fkq^n I$ for all $n\ge 1$. It follows from Lemma \ref{3.700} that
\begin{eqnarray*}
\ell(R/\fkq^{n+1})-\ell(R/I^{n+1})&=&\ell((\fkq^nI)/\fkq^{n+1})\\
&=& \ell((\fkq^n(\fkq :\m))/\fkq^{n+1})= \ell((\fkq^{n+1}:\fkm)/\fkq^{n+1})\end{eqnarray*}
for all $n\ge 0$.
Since $I^2=\fkq I$, we have $e_0(\fkq)=e_0(I)$. Therefore we have 
$$e_j(I)-e_j(\q)=  f_{j-1}(\q;R)$$
By Theorem \ref{P2.7}, we have 
\[e_j(I)-e_j(\q)=  f_{j-1}(\q;R) =  \left\{
\begin{array}{rl}
&(-1)^{d-1}(r_1(R)+r_0(R)), \quad \mbox{if $j=d$,} \\
& (-1)^{j-1}r_{d-j+1}(R) \hspace{1cm} \quad \mbox{otherwise.}
\end{array}
\right.\]

\end{proof}

\begin{cor}\label{4.300}
Suppose that $M$ is a sequentially Cohen-Macaulay $R$-module. Then  there exists an integer $n \gg 0$ such that for every distinguished parameter ideals $\q$ of $M$ contained in $\m^{n}$, one has  the equality
\[e_j(I)-e_j(\q)=  f_{j-1}(\q;R) =  \left\{
\begin{array}{rl}
&(-1)^{d-1}(r_1(R)+r_0(R)), \quad \mbox{if $j=d$,} \\
& (-1)^{j-1}r_{d-j+1}(R) \hspace{1cm} \quad \mbox{otherwise.}
\end{array}
\right.\]
\end{cor}
\begin{proof}
This is now immediate from Proposition \ref{coe} and Theorem \ref{3.800}.
\end{proof}

\medskip

\noindent

\begin{lem}\label{4.5}
Assume that $S$ is Cohen-Macaulay and 
$$[\q+\fka_{\ell-1}]:\fkm=\q:\m+\fka_{\ell-1}.$$
 Then 
$$ e_{j}(I;R/\fkb)-e_{j}(\fkq;R/\fkb)=
 \begin{cases}(-1)^{s}((e_{s+j}(I;R)-e_{s+j}(\q;R)))+r_d(R)&\text{if $j=1$,}
\\
(-1)^{s} (e_{s+j}(I;R)-e_{s+j}(\q;R))&\text{if } j\ge 2,\\
 \end{cases}$$  
where $s=d-d_{\ell-1}$.
\end{lem}
\begin{proof}
Since $[\q+\fka_{\ell-1}]:\fkm=\q:\m+\fka_{\ell-1}$,  we have $IS=\q S:\m S$. 
\begin{claim}\label{calim 1}
$(I^n+\fkb)\cap \fka_{\ell-1}=I^n\cap \fka_{\ell-1}$ for all $n$.
\end{claim}
\begin{proof}
Since $\fkq\subseteq \m^2$ and $S$ is Cohen-Macaulay, we have $(IS)^2 = (\q S)(IS)$ by   \cite[Theorem 3.7]{CHV}, so that $\operatorname{gr}_{IS}(S)$ is a Cohen-Macaulay ring. Therefore, we have $$I^nS:x_d = I^{n -1}S$$ for all $n \in \Bbb Z$. Consequently, $(I^n+\fka_{\ell-1}):x_d=I^{n-1}+\fka_{\ell-1}$.

Let $a\in (I^n+(x_d))\cap \fka_{\ell-1}$. Write $a=b+x_dc$ for $b\in I^n$ and $c\in R$. Then $c\in (I^n+\fka_{\ell-1}):x_d=I^{n-1}+\fka_{\ell-1}$. Thus since $\q$ is a distinguish parameter ideal, we have $x_dc\in x_dI^{n-1}+x_d\fka_{\ell-1}\subseteq I^n$. Therefore $a\in I^n\cap \fka_{\ell-1}$. Hence $I^n\cap \fka_{\ell-1}=(I^n+(x_d))\cap \fka_{\ell-1}$. By induction, we have
\begin{eqnarray*}
I^n\cap \fka_{\ell-1}&=&(I^n+(x_d))\cap \fka_{\ell-1}\\
&=&  \ldots = (I^n+\fkb)\cap \fka_{\ell-1},
\end{eqnarray*}
as required.
\end{proof}

 It follows from the above claim and the following exact sequences
$$0\to \fka_{\ell-1}/I^n\cap \fka_{\ell-1}\to R/I^n\to S/I^nS\to 0$$
for all $n\ge 0$ and 
$$0\to \fka_{\ell-1}/(\fkb+I^n)\cap \fka_{\ell-1}\to R/\fkb+I^n\to S/(\fkb+I^n)S\to 0$$
for all $n\ge 0$, that we have
$$\ell(R/I^n)-\ell(S/I^nS)=\ell(R/\fkb+I^n)-\ell(S/(\fkb+I^n)S)$$
Since $S$ is Cohen-Macaulay, by Lemma \ref{coe}, we have
$$\ell(S/I^nS)=e_0(I;S)\binom{n+d}{d}-r_d(S)\binom{n+d-1}{d-1}$$
$$\ell(S/(\fkb+I^n)S)=e_0(I;S)\binom{n+d_{\ell-1}}{d_{\ell-1}}-r_d(S)\binom{n+d_{\ell-1}-1}{d_{\ell}-1}$$
for all $n\ge 0$. Consequently, it follows on comparing the coefficients of the polynomials in the above equality  that 
$$ e_{j}(I;R/\fkb)=
 \begin{cases}(-1)^{s}e_{s+1}(I;R)+r_d(S) &\text{if $j=1$,}
\\
(-1)^{s}e_{s+j}(I;R)&\text{if } j\ge 2\\
 \end{cases}$$  
Similarly, we have $$ e_{j}(\q;R/\fkb)=(-1)^{s}e_{s+j}(I;R)$$  
for all $1\le j\le d_{\ell-1}$. It follows that
$$ e_{j}(I;R/\fkb)-e_{j}(\fkq;R/\fkb)=
 \begin{cases}(-1)^{s}((e_{s+j}(I;R)-e_{s+j}(\q;R)))+r_d(R)&\text{if $j=1$,}
\\
(-1)^{s} (e_{s+j}(I;R)-e_{s+j}(\q;R))&\text{if } j\ge 2.\\
 \end{cases}$$

\end{proof}

\begin{proposition}\label{P4.60}
 Assume that $d\ge 2$ and there exists  an integer $n$ such that for all distinguish parameter ideals $\fkq\subseteq \frak m^{n}$ we
    have
$$ e_1(I)-e_1(\fkq)\leq r_{d}(R).$$
Then $S$ is Cohen-Macaulay.
\end{proposition}
\begin{proof}
In the case in which $e_0(\m;R)=1$, we have  $e_0(\m;S)=1$, because $\dim \fka_{\ell-1}<\dim R$. And so the result in this case follws from $S$ is unmixed and Theorem 40.6 in \cite{Na}. Thus we suppose henceforth in this proof  that  $e_0(\m;R)>1$.

By Proposition \ref{exists}, there exists a Goto sequence $x_1,\ldots,x_{d-2}$ of type I in  $\m^n$.
Let  $\q_{d-2}=(x_1,\ldots,x_{d-2})$ and $A=R/\q_{d-2}$ and let $N$ denote the unmixed component of $A$. Then $A/N$ is a generalized Cohen-Macaulay ring since $\dim A/N=2$ and $A/N$ is unmixed. Therefore there exists an integer $n_0>n$  such that for all parameters  $x\in\m^{n_0}$, we have $r_1(A/(x)+N)=r_1(A/N)+r_2(A/N)$. Suppose that $d_{i_0}<d-2\le d_{i_0+1}$ for some $i_0$. 
Then $\Ann(\fka_{i_0})+\q_{d-2}$ is an $\m$-primary ideal of $R$. Then we can choose $x_{d-1}\in \m^{n_0}\cap \Ann(\fka_{i_0})$ as in Proposition \ref{exists element}.

Let $\q_{d-1}=(\q_{d-2},x_{d-1})$ and $B=R/\q_{d-1}$. Since $\dim B=1$, $B$ is sequentially Cohen-Macaulay ring. It follows from $\Ann \fka_{\ell-1}+\q_{d-1}$ is an $\m$-primary ideal of $R$ and Corollary \ref{4.300}, we have choose $x_d\in\Ann \fka_{\ell-1}$ such that $$e_1(x_dB:_B\m B;B)-e_1(x_dB;B)=r_1(B).$$


 Since $e(\m;R)>1$, by Proposition 2.3 in \cite{GSa1}, we get that $\m I^n = \m\q^n$ for all $n$. Therefore $I^n\subseteq \q^n:\m$ for all $n$. Put $G=\bigoplus\limits_{n\ge 0}\q^n/\q^{n+1}$ and $\fkM=\m/\q\oplus\bigoplus\limits_{n\ge 1}\q^n/\q^{n+1}$. 
Then we have $((0):\fkM)_n=[\q^n\cap (\q^{n+1}:\m)]/\q^{n+1}$ for all $n$. Thus
we have $((0):\fkM)_n=(\q^{n+1}:\m)/\q^{n+1}$ for all large $n$.
 Since $x_1,\ldots,x_d$ is a Goto sequence, by Lemma \ref{property}, $x_1,x_2,\ldots,x_d$ is $d$-sequence. Therefore, $x_1$ is a superficial element of $R$ with respect to $\q$. And so, we have $\q^{n+1}:x_1=\q^n$ for all $n\ge 0$. It follows from the exact sequences
$$0\to \frac{I^{n}\cap (\q^{n+1}:x_1)}{\q^{n}}\to \frac{I^{n}}{\q^{n}}\to \frac{I^{n+1}}{\q^{n+1}}\to \frac{I^{n+1}}{x_1I^n+\q^{n+1}}\to 0$$
$$\text{and} \ 0\to \frac{I^{n+1}\cap (x_1)+\q^{n+1}}{x_1I^n+\q^{n+1}}\to \frac{I^{n+1}}{x_1I^n+\q^{n+1}}\to \frac{I^{n+1}+(x_1)}{\q^{n+1}+(x_1)}\to 0$$
for 
all $n\ge 0$ that
 $$\ell(\frac{I^{n+1}}{\q^{n+1}})-\ell(\frac{I^{n}}{\q^{n}})=\ell(\frac{I^{n+1}}{x_1I^n+\q^{n+1}})\ge \ell(\frac{I^{n+1}+(x_1)}{\q^{n+1}+(x_1)}).$$
Consequenly, we have
$$e_1(I;R)-e_1(\q;R)\ge e_1(\q\overline{R}:\m\overline{R};\overline{R})-e_1(\q\overline{R};\overline{R})$$
because $\ell(R/\q^{n+1})-\ell(R/I^{n+1})=\ell(I^{n+1}/\q^{n+1})$ for all $n\ge 0$. Proceed inductively,  we have
$$ e_1(I;R)-e_1(\q;R)\ge e_1(x_dB:_B\m B;B)-e_1(x_dB;B)=r_1(B),$$
because of the choice of $x_d$.
However, since $x_1,\ldots,x_d$ is a Goto sequence, by Lemma \ref{property}, $\q$ is a distinguish parameter ideal and $ r_2(A)\ge r_d(R)$.
By hypothesis, $r_d(R)\ge e_1(I;R)-e_1(\q;R).$ Therefore, $ r_2(A)\ge r_1(B)$.
On the other hand, it follows from the exact sequence
$$0\to N\to A\to A/N\to 0$$
and $x_{d-1}$ is a regular of $A/N$ that $r_2(A/N)=r_2(A)$ and
$$0\to N/x_{d-1}N\to B\to A/(x_{d-1})+N\to 0.$$
Since $\dim N/x_{d-1}N=0$, we have $\H^1_\m(B)=\H^1_\m(A/(x_{d-1})+N)$, and so 
\begin{eqnarray*}
r_1(B)&=&r_1(A/(x_{d-1})+N)\\
&=&r_1(A/N)+r_2(A/N)=r_1(A/N)+r_2(A),
\end{eqnarray*}
because of the choice of $x_{d_1}$.
Therefore we have
 $r_1(A/N)=0$, and so $r_1(S/\q_{d-2}S)=0$. Hence $S$ is Cohen-Macaulay, because of Lemma \ref{property1}, and the proof is complete.
\end{proof}

The next corollary is now immediate.
\begin{cor}\label{C6.4}
  For all  integers $n$  there exists  a parameter ideal $\frak q\subseteq \frak m^n$, we
    have
$$ r(R)\leqslant e_1(I;R)-e_1(\fkq;R),$$
where $I=\fkq:\fkm$.
\end{cor}

The observation in Proposition \ref{P4.60} and Lemma \ref{4.5} provides a clue to a characterization of sequentially Cohen-Macaulay rings in terms of Hilbert coefficients of socle parameter ideals.

\begin{proposition}\label{P4.61}
Assume that there exists  an integer $n$ such that for all distinguish parameter ideals $\fkq\subseteq \frak m^{n}$ we
    have
$$ r_{j}(R)\ge (-1)^{d-j}(e_{d-j+1}(I)-e_{d-j+1}(\fkq)),$$
for all $2\le j\in\Lambda(R)$.
Then $R$ is sequentially Cohen-Macaulay.
\end{proposition}
\begin{proof}
 We use induction on the dimensional $d$ of $R$. In the case in which $\dim R=1$, it is clear that $R$ is sequentially Cohen-Macaulay. Suppose that $\dim R > 1$ and that our assertion holds true for $\dim R-1$.   Recall that $\fka_{\ell-1}$ is the unmixed component of $R$. Therefore, by the Prime Avoidance Theorem, we can choose the part of a system $x_{d_{\ell-1}+1},\ldots,x_d$ of parameters of $R$ such that $\fkb\subseteq \fkm^n$ and $\fkb \cap \fka_{\ell-1}=0$, where $\fkb=(x_{d_{\ell-1}+1},\ldots,x_d)$. Consequently, $(\fka_i+\fkb)/\fkb=\fka_i$ for all $i=0,\ldots,\ell-1$, and so $\Lambda(R)-\{d\}\subseteq \Lambda(R/\fkb)$.
 On the other hand, since $d\in\Lambda(R)$, we obtain $e_1(I;R)-e_1(\q)\le r_d(R)$ for all distinguished parameter ideals $\q\subseteq\m^n$. By Proposition \ref{P4.60}, $S$ is Cohen-Macaulay. It follows from  the  exact sequence
$$0\to \fka_{\ell-1}\to R \to S\to 0$$
that the sequence
$$0\to \fka_{\ell-1}\to R/\fkb \to S/\fkb S\to 0$$
is exact, and so $\Lambda(R/\fkb)=\Lambda(R)-\{d\}$.

Now let $x_1,\ldots,x_{d_{\ell-1}}$ be a distinguished system of parameters of $R/\fkb$. We show that $x_1,\ldots,x_d$ is a distinguished system of parameters of $R$. Indeed, let $d_{i}+1\le j\le d_{i+1}$ for some $i\not=\ell-1$. Since $d_i\in \Lambda(R/\fkb)$, $R/\fkb$ contain the  largest ideal $\fkc_i$ with $\dim \fkc_i=d_i$. Therefore $(\fka_i+\fkb)/\fkb\subseteq \fkc_i$. Since $x_j\fkc_i=0$, we obtain $x_j\fka_i\subseteq \fkb\cap \fka_i\subseteq \fkb\cap \fka_{\ell-1}=0$. Hence $x_j\fka_i=0$ for all $d_{i}+1\le j\le d_{i+1}$ and $i<\ell-1$. Hence system $x_1,\ldots,x_d$ of parameters is distinguished.

Put $\q=(x_1,\ldots,x_{d_{\ell-1}},\fkb)$ and assume that $\fkq\subseteq\m^n$.  It follows from Lemma \ref{4.5} that $$ e_{j}(I;R/\fkb)-e_{j}(\fkq;R/\fkb)=
 \begin{cases}(-1)^{s}((e_{s+j}(I;R)-e_{s+j}(\q;R)))+r_d(R)&\text{if $j=1$,}
\\
(-1)^{s} (e_{s+j}(I;R)-e_{s+j}(\q;R))&\text{if } j\ge 2,\\
 \end{cases}$$  
where $s=d-d_{\ell-1}$. However, it follows from $S$ is Cohen-Macaulay and the  exact sequence
$$0\to \fka_{\ell-1}\to R/\fkb \to S/\fkb S\to 0$$
that the following sequence
$$0\to \H^{d_{\ell-1}}_\m(\fka_{\ell-1})\to \H^{d_{\ell-1}}_\m(R/\fkb) \to \H^{d_{\ell-1}}_\m(S/\fkb S)\to 0$$
is exact. Moreover we have $\H^{d_{\ell-1}}_\m(R)\cong \H^{d_{\ell-1}}_\m(\fka_{\ell-1})$  
and $\H^i_\m(R)\cong \H^i_\m(\fka_{\ell-1})\cong \H^i_\m(R/\fkb)$ for all $i< d_{\ell-1}$. Thus we have $r_i(R)=r_i(R/\fkb)$ for all $i<d_{\ell-1}$ and 
$$r_{d_{\ell-1}}(R/\fkb)=r_{d_{\ell-1}}(R)+r_{d_{\ell-1}}(S/\fkb S)=r_{d_{\ell-1}}(R)+r_{d}(S).$$
Therefore  since $s+d_{\ell-1}-j=d-j$ we have
\begin{eqnarray*}
 r_{j}(R/\fkb)=r_{j}(R)&\ge&(-1)^{d-j}(e_{d-j+1}(I;R)-e_{d-j+1}(\q;R))\\
&= &(-1)^{d_{\ell-1}-j}(e_{d_{\ell-1}-j+1}(I;R/\fkb)-e_{d_{\ell-1}-j}(\fkq;R/\fkb)))
\end{eqnarray*}
for all $2\le j\in\Lambda(R/\fkb)-\{ d_{\ell-1}\}$.
Moreover, we have
\begin{eqnarray*}
r_{d_{\ell-1}}(R/\fkb)&=&r_{d_{\ell-1}}(R)+r_d(S)\\
&\ge&(-1)^{s}(e_{s+1}(I;R)-e_{s+1}(\q;R))+r_d(S)\\
&=&e_{1}(I;R/\fkb)-e_{1}(\fkq;R/\fkb)
\end{eqnarray*}
Consequently, $r_{d_{\ell-1}-j+1}(R/\fkb)\ge (-1)^{j+1}(e_{j}(I;R/\fkb)-e_{j}(\fkq;R/\fkb)))$ for all distinguished parameter ideals $\q\subseteq \m^n$ of $R/\fkb$ and $2\le j\in\Lambda(R/\fkb)$.
By the induction hypothesis, $R/\fkb$ is sequentially Cohen-Macaulay and so is also $R$. This completes the inductive step, and the
proof.

\end{proof}

Now, we can provide a characterization of sequentially Cohen-Macaulay rings.

\begin{theorem}\label{T5.3}
 The following statements are equivalent.
\begin{itemize}
\item[$(i)$] $R$ is sequentially Cohen-Macaulay.
\item[$(ii)$]  There exists an integer $n$ such that for all distinguish parameter ideals $\fkq\subseteq \fkm^n$ and $0\le i\le d-1$, we
    have
$$r_{d-i+1}(R)= (-1)^{i+1}(e_i(I)-e_i(\fkq)),$$
where $I=\fkq:\fkm$.

\item[$(iii)$]  There exists an integer $n$ such that for all good parameter ideals $\fkq\subseteq \fkm^n$ and $2\le j\in\Lambda(R)$, we
    have
$$ r_{j}(R)\ge(-1)^{d-j}(e_{d-j+1}(I)-e_{d-j+1}(\fkq)),$$
where $I=\fkq:\fkm$.

\end{itemize}

\end{theorem}

\begin{proof}
(1) $\Rightarrow$ (2) This is now immediate from Theorem \ref{3.800} and Proposition \ref{coe}.

(2) $\Rightarrow$ (3) This is obvious.

(3) $\Rightarrow$ (1) This now immediate from  Proposition \ref{P4.61}.

\end{proof}

\section{Noetherian coefficients}


Continuing our discussion in last section, we will see how the sequentially Cohen-Macaulayness of rings is relation to the Noetherian coefficients of the socle of distinguished parameter ideals. To discuss it, we need a relationship between Chern coefficient and the irreducible multiplicity.

\begin{setting}\label{2.6}{\rm
Assume that $R$ is a homomorphic image of a Cohen-Macaulay local ring. Let $\calD = \{\fka_i\}_{0 \le i \le \ell}$ be the dimension filtration of $R$ with $\dim \fka_i=d_i$. We put  $S = R/\fka_{\ell -1}$ and  choose  a distinguished system $x_1, x_2, \ldots, x_d$ of parameters of $R$. Put $\q=(x_1,x_2,\ldots,x_d)$, $\fkb=(x_{d_{\ell-1}+1},\ldots,x_d)$ and $I=\q:\m$.
}
\end{setting}

\begin{lem}\label{5.1}
Assume that $e_0(\m;R)>1$. Then for all parameter ideal $\q$, we have
$$e_1(I;R)-e_1(\q;R)\le f_0(R),$$
where $I=\q:\m$.
\end{lem}
\begin{proof}
Since $e_0(\m;R)>1$, by Proposition 2.3 in \cite{GSa1}, we get that $\m I^n = \m\q^n$ for all $n$. Therefore $I^n\subseteq \q^n:\m$ for all $n$. Consequence, we obtain
$$\ell(R/\fkq^{n+1})-\ell(R/I^{n+1})=\ell(I^{n+1}/\fkq^{n+1})\le \ell((\q^{n+1}:\m/\q^{n+1}).$$ 
But this means that $e_1(I;R)-e_1(\q;R)\le f_0(\q;R)$. 
\end{proof}

\begin{prop}
 Assume that  $R$ is unmixed. Then for all parameter ideals $\q\subseteq \m^2$, we have
$$e_1(\q:\m)-e_1(\q)\le f_0(\q).$$
\end{prop}
\begin{proof}
Our result in the case in which $e_0(\m;R)>1$ is immediate from  Lemma \ref{5.1}. Thus we suppose henceforth in this proof that $e_0(\m;R)=1$. It follows from  $R$ is unmixed and Theorem 40.6 in \cite{Na} that $R$ is Cohen-Macaulay.  
Since $R$ is unmixed, every parameter ideals $\q$ are distinguished.  Therefore by Theorem \ref{3.800}, there exists an integer $n$ such that for all parameter ideals $\q$, we have
 $$\calN(\q; R)=\sum\limits_{j \in \Bbb Z}r_j(R).$$
It follows from Proposition \ref{coe}, we have
$$ e_1(\q:\m)-e_1(\fkq)=f_0(\fkq;R),$$
and the proof is complete.
\end{proof}

\begin{cor}\label{c5.2}
 Assume that $d\ge 2$ and there exists  an integer $n$ such that for all distinguish parameter ideals $\fkq\subseteq \frak m^{n}$ we
    have
$$ f_0(\q;R)\leq r_{d}(R).$$
Then $S$ is Cohen-Macaulay.
\end{cor}
\begin{proof}
In the case in which $e_0(\m;R)=1$, we have  $e_0(\m;S)=1$, because $\dim \fka_{\ell-1}<\dim R$. And so the result in this case follws from $S$ is unmixed and Theorem 40.6 in \cite{Na}. Thus we suppose henceforth in this proof  that  $e_0(\m;R)>1$.  By Lemma \ref{5.1}, for all distinguish parameter ideals $\fkq\subseteq \frak m^{n}$ we
    have
$$e_1(I;R)-e_1(\q;R)\le f_0(R)\le r_d(R).$$
It follows from Proposition \ref{P4.60} that $S$ is Cohen-Macaulay. 
\end{proof}

 The next corollary is now immediate.
\begin{cor}\label{C5.5}
  For all  integers $n$  there exists  a parameter ideal $\frak q\subseteq \frak m^n$, we
    have
$$ r_d(R)\leqslant f_0(\fkq;R).$$

\end{cor}

\begin{lem}
Assume that $S$ is Cohen-Macaulay and 
$$[\q+\fka_{\ell-1}]:\fkm=\q:\m+\fka_{\ell-1}.$$
 Then
$$ f_j(\fkq;R/\fkb)=
 \begin{cases}(-1)^{s}f_{s+j}(\q;R)+r_d(R)&\text{if $j=0$,}
\\
(-1)^{s} f_{s+j}(\q;R)&\text{if } j\ge 1,\\
 \end{cases}$$
where $s=d-d_{\ell-1}$.
\end{lem}
\begin{proof}
Since $S$ is Cohen-Macaulay and $\q$ is a parameter ideal of $S$, we have
$$0\to \fka_{\ell-1}/\fkq^n\fka_{\ell-1}\to R/\q^n\to S/\q^nS\to 0.$$
It follows from  $[\q^n+\fka_{\ell-1}]:\fkm=\q^n:\m+\fka_{\ell-1}$, by the lemma \ref{2.101}, that by
applying $\Hom_R(k,)$, we obtain the following exact sequence
$$0\to\frac{\fkq^n\fka_{\ell-1}:\m}{\q^n\fka_{\ell-1}}\to \frac{\fkq^n:\m}{\fkq^n}\to \frac{\fkq^nS:\m}{\fkq^nS}\to 0.$$
Since $S$ is Cohen-Macaulay and $\fkb$ is an ideal generated by a part system of parameters of $S$, we have
$$0\to\frac{\fka_{\ell-1}}{\q^n\fka_{\ell-1}}\to \frac{R}{\fkq^n+\fkb}\to \frac{S}{\fkq^nS+\fkb S}\to 0$$
It follows from  $[\q^n+\fka_{\ell-1}+\fkb]:\fkm=\q^n:\m+\fka_{\ell-1}+\fkb$, by the lemma \ref{2.101}, that by
applying $\Hom_R(k,)$, we obtain the following exact sequence
$$0\to\frac{\fkq^n\fka_{\ell-1}:\m}{\q^n\fka_{\ell-1}}\to \frac{(\fkq^n+\fkb):\m}{\fkq^n+\fkb}\to \frac{(\fkq^nS+\fkb S):\m}{\fkq^nS+\fkb S}\to 0.$$
From the above exact sequences, we have
$$\ell(\frac{\fkq^n:\m}{\fkq^n})-\ell(\frac{\fkq^n S:\m}{\fkq^n S})=\ell(\frac{(\fkq^n+\fkb ):\m}{\fkq^n+\fkb})-\ell(\frac{(\fkq^nS+\fkb S):\m}{\q^n S+\fkb S}).$$
Since $S$ is Cohen-Macaulay, by Proposition \ref{P2.7}, we have $\ell(\frac{\q^{n+1}S:\m}{\q^{n+1}S})=r_d(S)\binom{n+d-1}{d-1}$ and  $\ell(\frac{(\q^{n+1}S+\fkb S):\m}{\q^{n+1}S+\fkb S})=r_d(S)\binom{n+d_{\ell-1}-1}{d_{\ell-1}-1}$.
Therefore since $d_{\ell-1}<d$ and $r_d(S)=r_d(R)$, we have
$$ f_j(\fkq;R/\fkb)=
 \begin{cases}(-1)^{s}f_{s+j}(\q;R)+r_d(R)&\text{if $j=0$,}
\\
(-1)^{s} f_{s+j}(\q;R)&\text{if } j\ge 1,\\
 \end{cases}$$
where $s=d-d_{\ell-1}$.

\end{proof}

\begin{proposition}\label{P5.61}
Assume that there exists an integer $n$ such that for all distinguished parameter ideals $\fkq\subseteq \fkm^n$ and $2\le j\in\Lambda(R)$, we
    have
$$ (-1)^{d-j}f_{d-j}(\q;R)\leq r_{j}(R).$$
Then $R$ is sequentially Cohen-Macaulay.
\end{proposition}
\begin{proof}
We argue by induction on the dimensional $d$ of $R$, the result being clear in the case
in which $d=1$. Suppose, inductively, that $d > 1$ and the result has been
proved for smaller values of $d$.

 Recall that $\fka_{\ell-1}$ is the unmixed component of $R$. Therefore, by the Prime Avoidance Theorem, we can choose the part of a system $x_{d_{\ell-1}+1},\ldots,x_d$ of parameters of $R$ such that $\fkb\subseteq \fkm^n$ and $\fkb \cap \fka_{\ell-1}=0$, where $\fkb=(x_{d_{\ell-1}+1},\ldots,x_d)$. Consequently, $(\fka_i+\fkb)/\fkb=\fka_i$ for all $i=0,\ldots,\ell-1$, and so $\Lambda(R)-\{d\}\subseteq \Lambda(R/\fkb)$.
 On the other hand, since $d\in\Lambda(R)$, we obtain $f_0(\fkq;R)\le r_d(R)$ for all distinguished parameter ideals $\q\subseteq\m^n$. By Corollary \ref{c5.2}, $S$ is Cohen-Macaulay. It follows from  the  exact sequence
$$0\to \fka_{\ell-1}\to R \to S\to 0$$
that the sequence
$$0\to \fka_{\ell-1}\to R/\fkb \to S/\fkb S\to 0$$
is exact, and so $\Lambda(R/\fkb)=\Lambda(R)-\{d\}$.

Now let $x_1,\ldots,x_{d_{\ell-1}}$ be a distinguished system of parameters of $R/\fkb$. We show that $x_1,\ldots,x_d$ is a distinguished system of parameters of $R$. Indeed, let $d_{i}+1\le j\le d_{i+1}$ for some $i\not=\ell-1$. Since $d_i\in \Lambda(R/\fkb)$, $R/\fkb$ contain the  largest ideal $\fkc_i$ with $\dim \fkc_i=d_i$. Therefore $(\fka_i+\fkb)/\fkb\subseteq \fkc_i$. Since $x_j\fkc_i=0$, we obtain $x_j\fka_i\subseteq \fkb\cap \fka_i\subseteq \fkb\cap \fka_{\ell-1}=0$. Hence $x_j\fka_i=0$ for all $d_{i}+1\le j\le d_{i+1}$ and $i<\ell-1$. Hence system $x_1,\ldots,x_d$ of parameters is distinguished.

Put $\q=(x_1,\ldots,x_{d_{\ell-1}},\fkb)$ and assume that $\fkq\subseteq\m^n$.  It follows from Lemma \ref{4.5} that 
$$ f_j(\fkq;R/\fkb)=
 \begin{cases}(-1)^{s}f_{s+j}(\q;R)+r_d(R)&\text{if $j=0$,}
\\
(-1)^{s} f_{s+j}(\q;R)&\text{if } j\ge 1,\\
 \end{cases}$$
where $s=d-d_{\ell-1}$. However, it follows from $S$ is Cohen-Macaulay and the  exact sequence
$$0\to \fka_{\ell-1}\to R/\fkb \to S/\fkb S\to 0$$
that the following sequence
$$0\to \H^{d_{\ell-1}}_\m(\fka_{\ell-1})\to \H^{d_{\ell-1}}_\m(R/\fkb) \to \H^{d_{\ell-1}}_\m(S/\fkb S)\to 0$$
is exact. Moreover we have $\H^{d_{\ell-1}}_\m(R)\cong \H^{d_{\ell-1}}_\m(\fka_{\ell-1})$  
and $\H^i_\m(R)\cong \H^i_\m(\fka_{\ell-1})\cong \H^i_\m(R/\fkb)$ for all $i< d_{\ell-1}$. Thus we have $r_i(R)=r_i(R/\fkb)$ for all $i<d_{\ell-1}$ and 
$$r_{d_{\ell-1}}(R/\fkb)=r_{d_{\ell-1}}(R)+r_{d_{\ell-1}}(S/\fkb S)=r_{d_{\ell-1}}(R)+r_{d}(S).$$
Therefore  since $s+d_{\ell-1}-j=d-j$ we have
\begin{eqnarray*}
 r_{j}(R/\fkb)=r_{j}(R)&\ge&(-1)^{d-j}f_{d-j}(\q;R)=(-1)^{d_{\ell-1}-j}f_{d_{\ell-1}-j}(\q;R/\fkb)
\end{eqnarray*}
for all $2\le j\in\Lambda(R/\fkb)-\{ d_{\ell-1}\}$.
Moreover, we have
\begin{eqnarray*}
r_{d_{\ell-1}}(R/\fkb)&=&r_{d_{\ell-1}}(R)+r_d(S)\ge(-1)^{s}f_s(\fkq;R)+r_d(S)=f_0(\fkq;R/\fkb).
\end{eqnarray*}
Consequently, $$ (-1)^{d_{\ell-1}-j}f_{d_{\ell-1}-j}(\q;R/\fkb)\leq r_{j}(R/\fkb).$$
 for all distinguished parameter ideals $\q\subseteq \m^n$ of $R/\fkb$ and $2\le j\in\Lambda(R/\fkb)$.
Application of the inductive hypothesis to the ring $R/\fkb$ shows that $R/\fkb$ is sequentially Cohen-Macaulay and so is also $R$. This completes the inductive step, and the proof.

\end{proof}

We close this section with the following, which   will be used in our discussion of regular local
rings  in the next section.

\begin{theorem}\label{5800}
 The following statements are equivalent.

\begin{itemize}
\item[$(i)$] $R$ is a  sequentially Cohen-Macaulay  $R$-module.
\item[$(ii)$]  There exists an integer $n$ such that for all distinguished parameter ideals $\fkq\subseteq \fkm^n$ and $j=0,\ldots,d-2$, we
    have
$$ (-1)^{j}f_{j}(\q;R)= r_{d-j}(R).$$

\item[$(iii)$]  There exists an integer $n$ such that for all distinguished parameter ideals $\fkq\subseteq \fkm^n$ and $2\le j\in\Lambda(R)$, we
    have
$$ (-1)^{d-j}f_{d-j}(\q;R)\leq r_{j}(R).$$

\end{itemize}

\end{theorem}

\begin{proof}
(1) $\Rightarrow$ (2) This is now immediate from Theorem \ref{3.800} and Proposition \ref{P2.7}.

(2) $\Rightarrow$ (3) This is obvious.

(3) $\Rightarrow$ (1) This now immediate from  Proposition \ref{P5.61}.

\end{proof}

\section{Chern coefficients}
In this section, we are going to discuss  the characterizations of the regularity,  Gorensteinness and Cohen-Macaulayness  of local rings. Probably the most important applications of Theorem \ref{5800} and \ref{T5.3} can be summarized in this section.

\begin{setting}\label{2.6}{\rm
Assume that $R$ is a homomorphic image of a Cohen-Macaulay local ring. Let $\calD = \{\fka_i\}_{0 \le i \le \ell}$ be the dimension filtration of $R$ with $\dim \fka_i=d_i$.
 }
\end{setting}

\begin{theorem}\label{6.3}
 $R$ is Gorenstein if and only if for all parameter ideals $\fkq\subseteq \fkm^2$ and $n\gg 0$, we
    have
$$ \calN(\q^{n+1};R)=\binom{n+d-1}{d-1}.$$

\end{theorem}
\begin{proof}
  {\it Only if}: Since $R$ is Gorenstein, by Proposition \ref{P2.7}, we have  
$$\mathcal{N}(\frak q^{n+1};R) = \sum\limits_{i=1}^dr_i(R)\binom{n+i-1}{i-1}+r_0(R)$$
for all $n\ge 1$. On the other hand, we have $r_i(R)=0$ for all $i< d$ and $r_d(R)=1$, because $R$ is Gorenstein. Hence
$$\mathcal{N}(\frak q^{n+1};R) = \binom{n+d-1}{d-1}.$$
for all $n\ge 1$.\\ 
  {\it If}: By the hypothesis, for all distinguished parameter ideals $\q\subseteq \m^2$ we have
$$(-1)^jf_j(\q;R)\le r_{d-i}(R).$$
It follows from Theorem \ref{5800} that $R$ is sequentially Cohen-Macaulay. By Theorem \ref{3.800}, there exists a distinguished parameter ideals $\q\subseteq \m^2$ such that $\calN(\q;R)=\sum\limits_{j\in \Bbb Z}r_j(R)$. Apply the Proposition \ref{P2.7} to this equation to
obtain that
$$\mathcal{N}(\frak q^{n+1};R) = \sum\limits_{i=1}^dr_i(R)\binom{n+i-1}{i-1}+r_0(R).$$
Since $\mathcal{N}(\frak q^{n+1};R) = \binom{n+d-1}{d-1}$, we have $r_i(R)=0$ for all $i\le d-1$ and $r_d(R)=1$. Hence $R$ is Gorenstein.
\end{proof}

In \cite[Theorem 5.2]{CQT}, N. T. Cuong, P. H. Quy and first author showed that $R$ is Cohen-Macaulay if and only if for all parameter ideals $\fkq\subseteq \fkm^2$ and $n\ge 0$, we
    have
$ \calN(\q^{n+1};R)=r_d(R)\binom{n+d-1}{d-1}.$ Note that the condition of Hilbert function $\calN(\q^{n+1};R)$, holding true for all $n\ge 0$, is necessary to their proof. 
The result of  Theorem 5.2 in \cite{CQT} was actually covered in the following result,
but in view of the importance of the following result we changed the condition from Hilbert function to Hilbert polynomial.

\begin{theorem}\label{6.3}
 $R$ is Cohen-Macaulay if and only if for all parameter ideals $\fkq\subseteq \fkm^2$ and $n\gg 0$, we
    have
$$ \calN(\q^{n+1};R)=r_d(R)\binom{n+d-1}{d-1}.$$

\end{theorem}
\begin{proof}
  {\it Only if}: Since $R$ is Cohen-Macaulay, by Proposition \ref{P2.7}, we have  
$$\mathcal{N}(\frak q^{n+1};R) = \sum\limits_{i=1}^dr_i(R)\binom{n+i-1}{i-1}+r_0(R)$$
for all $n\ge 1$. On the other hand, we have $r_i(R)=0$ for all $i< d$, because $R$ is Cohen-Macaulay. Hence
$$\mathcal{N}(\frak q^{n+1};R) = r_d(R)\binom{n+d-1}{d-1}.$$
for all $n\ge 1$.\\ 
  {\it If}: By the hypothesis, for all distinguished parameter ideals $\q\subseteq \m^2$ we have
$$(-1)^jf_j(\q;R)\le r_{d-i}(R).$$
It follows from Theorem \ref{5800} that $R$ is sequentially Cohen-Macaulay. By Theorem \ref{3.800}, there exists a distinguished parameter ideals $\q\subseteq \m^2$ such that $\calN(\q;R)=\sum\limits_{j\in \Bbb Z}r_j(R)$. Apply the Proposition \ref{P2.7} to this equation to
obtain that
$$\mathcal{N}(\frak q^{n+1};R) = \sum\limits_{i=1}^dr_i(R)\binom{n+i-1}{i-1}+r_0(R).$$
Since $\mathcal{N}(\frak q^{n+1};R) = r_d(R)\binom{n+d-1}{d-1}$, we have $r_i(R)=0$ for all $i\le d-1$. Hence $R$ is Cohen-Macaulay.
\end{proof}

In the sequel, we shall only use the following.

\begin{setting}\label{2.6}{\rm
Let $R$ be a Noetherian local ring with maximal ideal $\frak m$, $d=\dim R\geqslant 2$. Assume that  $R$ is unmixed, that is $\dim \hat R/\fkp=d$ for all $\fkp\in \Ass(\hat R)$. }
\end{setting}

\begin{theorem}\label{6.1}
 $R$ is regular if and only if $f_0(\m)=1.$
\end{theorem}
\begin{proof}
 Since $R$ is regular, $R$ is Cohen-Macaulay and $\m$ is a parameter ideal of $R$. Since $R$ is unmxied, $\m$ is a distinguished parameter ideal of $R$. Thus by Theorem \ref{5800}, we have
$f_0(\m;R)=r_d(R)$. But $r_d(R)=1$ because $R$ is regular. Hence $f_0(\m;R)=1$.

Conversely, since $R$ is unmixed, we have $\depth(R)> 0$. Therefore there exists an integer $n_0$ such that
$\m^{n+1}:\m=\m^n$ for all $n\ge n_0$. It follows that $e_0(\m;R)=f_0(\m;R)=1$. Since $R$ is unmixed, by Theorem 40.6 in \cite{Na}, $R$ is regular.

\end{proof}

Theorem \ref{5800} and \ref{T5.3} have a very important corollary. One obvious consequence
is the following

\begin{theorem}\label{6.2}
 The following statements are equivalent.

\begin{itemize}
\item[$(1)$] $R$ is Gorenstein.
\item[$(2)$]  For all parameter ideals $\frak q\subseteq \m^2$, we
    have
$$ f_0(\fkq;R)=1.$$

\item[$(3)$]  For all parameter ideals $\frak q\subseteq \m^2$, we
    have
$$ e_1(I)-e_1(\fkq)\le 1,$$
where $I=\fkq:\fkm$.

\end{itemize}

\end{theorem}

\begin{proof}
(1) $\Rightarrow$ (2) Since $R$ is Gorenstein, $R$ is Cohen-Macaulay. Let  $\q$ be a parameter ideal of $R$. Then we have $\calN(\q; R)=\sum\limits_{j \in \Bbb Z}\ell_R((0):_{\H^j_\fkm(R)}\fkm).$
Since $R$ is unmixed, parameter ideal $\q$ is distinguished. Then by Proposition \ref{P2.7}, we have
 $f_0(\q;R)=r_d(R)$. But $r_d(R)=1$ because $R$ is Gorenstein. Hence $f_0(\q;R)=1$.

(2) $\Rightarrow$ (3) 
Our result in the case in which $e_0(\m;R)>1$ is immediate from  Lemma \ref{5.1}. Thus we suppose henceforth in this proof that $e_0(\m;R)=1$. It follows from  $R$ is unmixed and Theorem 40.6 in \cite{Na} that $R$ is Cohen-Macaulay.  
Since $R$ is unmixed, every parameter ideals $\q$ are distinguished.  Therefore by Theorem \ref{3.800}, there exists an integer $n$ such that for all parameter ideals $\q$, we have
 $$\calN(\q; R)=\sum\limits_{j \in \Bbb Z}r_j(R).$$
It follows from Proposition \ref{coe}, we have
$$ e_1(I)-e_1(\fkq)=f_0(\fkq;R)=1.$$

(3) $\Rightarrow$ (1) Since $ e_1(I)-e_1(\fkq)\le1$ for all parameter ideals $\q$, by Proposition \ref{P4.60} $R$ is Cohen-Macaulay. Since $R$ is unmixed, every parameter ideals $\q$ are distinguished.  Therefore by Theorem \ref{3.800}, there exists an integer $n$ such that for all parameter ideals $\q$, we have
$\calN(\q; R)=\sum\limits_{j \in \Bbb Z}r_j(R).$
It follows from Proposition \ref{coe}, we have
$$ r_d(R)=e_1(I)-e_1(\fkq)=1.$$
Hence $R$ is Gorenstein.

\end{proof}

Our next result establishes some absolutely fundamental facts about
 Cohen-Macaulay rings.

\begin{theorem}\label{6.3}
 The following statements are equivalent.

\begin{itemize}
\item[$(1)$] $R$ is Cohen-Macaulay.
\item[$(2)$]  For all parameter ideals $\fkq\subseteq \fkm^2$, we
    have
$$ f_0(\fkq;R)=r_d(R).$$


\item[$(3)$]  For all parameter ideals $\fkq\subseteq \fkm^2$, we
    have
$$ e_1(I;R)-e_1(\q;R)\le r_d(R),$$
where $I=\fkq:\fkm$.

\end{itemize}

\end{theorem}

\begin{proof}
(1) $\Rightarrow$ (2)  Let  $\q$ be a parameter ideal of $R$. Then we have $\calN(\q; R)=\sum\limits_{j \in \Bbb Z}\ell_R((0):_{\H^j_\fkm(R)}\fkm)$, because $R$ is Cohen-Macaulay.
Since $R$ is unmixed, parameter ideal $\q$ is distinguished. Then by Proposition \ref{P2.7}, we have
 $f_0(\q;R)=r_d(R)$.

(2) $\Rightarrow$ (3) 
In the case in which $e_0(\m;R)>1$ there is nothing to prove, because of Lemma \ref{5.1} and so we suppose that $e_0(\m;R)=1$. Then  by Theorem 40.6 in \cite{Na}, $R$ is Cohen-Macaulay because $R$ is unmixed.  Let  $\q$ be a parameter ideal of $R$ such that $\q\subseteq \m^2$ and put $I=\fkq:\fkm$. Then we have $$\calN(\q; R)=\sum\limits_{j \in \Bbb Z}\ell_R((0):_{\H^j_\fkm(R)}\fkm),$$
since $R$ is Cohen-Macaulay.   It follows from Proposition \ref{coe} and $\q\subseteq\m^2$, we have
$$ e_1(I)-e_1(\fkq)=f_0(\fkq;R)=r_d(R).$$

(3) $\Rightarrow$ (1) Since $ e_1(I)-e_1(\fkq)\le r_d(R)$ for all parameter ideals $\q\subseteq\m^2$, by Proposition \ref{P4.60} $R$ is Cohen-Macaulay, as required.

\end{proof}


\end{document}